\documentclass[preprint]{imsart}

\usepackage{amsthm,amsmath,natbib}
\usepackage{amsfonts}

\usepackage{color}

\startlocaldefs

\newtheorem{theorem}{Theorem}
\newtheorem{lemma}{Lemma}

\newtheorem{remark}{Remark}

\newcommand{\Beta}{\rm{B}}

\newcommand{\bx}{{\bf x}}

\def\be{\begin{equation}}
\def\ee{\end{equation}}
\def\beq{\begin{eqnarray*}}
\def\eeq{\end{eqnarray*}}
\def\bea{\begin{eqnarray}}
\def\eea{\end{eqnarray}}

\def\Var{{\rm Var}}
\def\var{{\rm Var}}
\def\Cov{\rm Cov}

\newcommand{\reals}{{\mathbb R}}
\newcommand{\bbr}{\reals}
\newcommand{\bbz}{{\mathbb Z}}

\newcommand{\bbn}{{\mathbb N}}
\newcommand{\eid}{\overset{d}{=}}
\newcommand{\vep}{\varepsilon}

\newcommand{\one}{{\bf 1}}

\endlocaldefs

\begin{document}

\begin{frontmatter}

%
%
%
\title{Clustering of large deviations in moving average processes: the long memory regime}
\runtitle{Clustering of large deviations}


\author{\fnms{Arijit} \snm{Chakrabarty}\ead[label=e1]{arijit.isi@gmail.com}}
\address{Theoretical Statistics and Mathematics Unit \\
  Indian Statistical Institute,  Kolkata \\
\printead{e1}}
\and
\author{\fnms{Gennady} \snm{Samorodnitsky} \thanksref{t1} \ead[label=e2]{gs18@cornell.edu}}
  \thankstext{t1}{$^*$The corresponding author.  Research partially 
   supported by NSF grant  DMS-2015242 at Cornell University.  Part of
   this work was performed when GS was visiting Department of
   Mathematics of National University of Singapore, whose hospitality
   is gratefully acknowledged.} 
 \address{School of Operations Research\\
   and Information Engineering\\
  Cornell University\\
\printead{e2}}
\affiliation{Indian Statistical Institute, Kolkata and Cornell University}

\runauthor{Chakrabarty and Samorodnitsky}
  
\numberwithin{equation}{section}
\numberwithin{theorem}{section}


\begin{abstract}
 We investigate how  large deviations events cluster in the framework of an
infinite moving average process with light-tailed noise and long
memory. The long memory makes clusters larger, and the asymptotic
behaviour of the size of the cluster turns out to be described by the
first hitting time of a randomly shifted fractional Brownian motion
with drift.  
\end{abstract}

\begin{keyword}[class=AMS]
\kwd[Primary ]{60F10}
\end{keyword}

\begin{keyword}
  \kwd{large deviations}
   \kwd{clustering}
   \kwd{infinite moving average}
   \kwd{long memory}
 \end{keyword}

\end{frontmatter}

%
%

\section{Introduction} \label {sec:intro}

We consider an infinite moving average process of the form 
\begin{equation} \label{e:MA}
X_n=\sum_{i=0}^\infty a_i Z_{n-i}\,,\ n\ge0\,,
\end{equation}
where the noise variables $(Z_n:n\in\bbz)$ are assumed to bef i.i.d.\
non-degenerate random variables. The noise distribution $F_Z$ is
assumed have finite exponential moments: 
\begin{equation}  \label{eq.exp}
  \int_\bbr e^{tz}\, F_Z(dz)<\infty\text{ for all }t\in\bbr\,.
\end{equation}
Furthermore, assuming that the noise is centred: 
\begin{equation} \label{ima.eq1}
  \int_\bbr z\, F_Z(dz) =0\,, 
\end{equation}
the series defining the process in \eqref{e:MA} converges  if and
only if the coefficients $a_0,a_1,a_2\ldots$ satisfy 
\begin{equation}   \label{eq.sqsum}
  \sum_{j=0}^\infty a_j^2<\infty\,.
\end{equation}
In this case $(X_n)$ is 
 a zero mean stationary ergodic 
process. For $\vep>0$ we consider the sequence of large deviation
events 
\begin{equation} \label{e:E.j}
E_j(n,\vep)=\left\{\frac1{n}\sum_{i=j}^{n+j-1}X_i\ge \vep\right\}, \
j\geq 0. 
\end{equation}
By stationarity, each event $E_j(n,\vep)$ is equally rare,
and we are interested in the cluster of these events that occur given
that the event $E_0(n,\vep)$ occurs.

In \cite{chakrabarty:samorodnitsky:2022} the short memory case was
considered. In this context, ``short memory'' corresponds to the case
\begin{equation} \label{e:short.m}
\sum_{n=0}^\infty |a_n|<\infty \ \text{and} \ \sum_{n=0}^\infty
a_n\not= 0.
\end{equation}
In this short memory case the conditional on $E_0(n,\vep) $ law of the
sequence  $\bigl( \one(  E_j(n,\vep), \,j=1,2,\ldots)$ converges
weakly, as $n\to\infty$, to the law of a sequence with a.s. finitely
many non-zero entries. the total number $D_\vep$ of the non-zero
entries turns out to scale as $\vep^{-2}$, and $\vep^2D_\vep$ has an
interesting weak limit as $\vep\to 0$. We refer the reader to
\cite{chakrabarty:samorodnitsky:2022} for details, and a minor
technical condition required for the above statements.

In the present paper we are interested in the long memory case. For
the moving average processes \eqref{e:MA} ``long memory'' refers to
the case when the coefficients $(a_j)$ satisfy the square summability
assumption \eqref{eq.sqsum} but not the absolute summability
assumption in \eqref{e:short.m}. A typical assumption in this is
\begin{equation} \label{e:regvar.a}
(a_n) \ \text{is regularly varying with exponent} \ -\alpha, \
1/2<\alpha<1;
\end{equation}
see \cite{samorodnitsky:2016}. It turns out that, in this case (under
certain technical assumptions, an example of which is below), the
conditional  on $E_0(n,\vep) $ law of the
sequence  $\bigl( \one(  E_j(n,\vep), \,j=1,2,\ldots)$ converges
weakly, as $n\to\infty$, to the degenerate probability measure
$\delta_{(1,1,\ldots)}$. That is, once the event $E_0(n,\vep) $
occurs, the events $(E_j(n,\vep))$ become very likely. In order to
understand their structure we concentrate on the random variables 
\begin{equation} \label{e:I.n}
I_n(\vep)=\inf\left\{j\ge1:E_j(n,\vep)\text{ does not
    occur}\right\},\, n\ge1
\end{equation} 
and establish a weak limit theorem for this sequence under a proper
scaling. Interestingly, the limit turns out to be the law of the first hitting
time of a randomly shifted fractional Brownian motion with drift.  


The main result containing the above limit theorem and the technical
assumptions it requires are in Section \ref{sec:main}. The proof of
the main theorem requires a long sequence of preliminary results, all
of which are presented in that section. Finally, some useful facts
needed for the proofs are collected in Section \ref{sec:useful.f}.

\section{The assumptions and the main result} \label{sec:main} 

Our result on clustering of large deviation events in the long memory
case will require a number of assumptions that we state next. First of
all, we will replace the assumption of regular variation
\eqref{e:regvar.a} by the asymptotic power function assumption 
\begin{equation} \label{e:power.a}
a_n\sim n^{-\alpha}, \ 
1/2<\alpha<1, \  \text{and is eventually monotone}.
\end{equation}
There is no doubt that the results of the paper hold under the more
general regular variation assumption as well. The extra generality
will, however, require making an already highly technical argument
even more so. The potentially resulting lack of clarity makes the
added generality less valuable. The same is true about the eventual
monotonicity assumption. 

We will need additional assumptions on the distribution of the noise
variables. We will assume that some $\theta_0>0$,
\begin{equation}\label{eq.chf}
\sup_{|\theta|\le\theta_0}\int_{-\infty}^\infty
t^2\left|\int_{-\infty}^\infty e^{(i t+\theta)z}\, F_Z(dz)\right|
dt<\infty\,. 
\end{equation}

Next, let 
\begin{equation} \label{e:noise.v}
\sigma_Z^2=  \int_\bbr z^2\, F_Z(dz) 
\end{equation}
be the variance of the noise. Denote
\begin{equation}   \label{eq.kappa}
  \kappa = \text{the smallest integer} >\frac{4\alpha-1}{2-2\alpha}. 
 \end{equation}
In other words, $\kappa=\bigl[ (1+2\alpha)/(2-2\alpha)\bigr]$. We
assume that a generic noise variable $Z$ satisfies 
\begin{equation} \label{e:moment.eq}
  EZ^i = EG^i \ \text{for} \ 1\le i\le\kappa,
\end{equation} 
where $G\sim N(0,\sigma_Z^2)$.

\begin{remark} \label{rk:cond}
  {\rm
It is standard to verify that \eqref{eq.chf} implies that the noise
distribution has a twice continuously differentiable density
$f_Z$. One the other hand, a sufficient condition for \eqref{eq.chf} is that
the noise distribution has a four times continuously differentiable density  $f_Z$ such that
\[
 \int_{-\infty}^\infty e^{\theta_0
   |x|}\left|\frac{d^i}{dx^i}f_Z(x)\right|dx<\infty\ \text{for} \ i=1,2,3,4.
\]

The moment equality assumption \eqref{e:moment.eq} restricts how far
the the noise distribution can be from a normal distribution. Note
that in the range $1/2<\alpha<5/8$ we have $\kappa=2$,  in which case
the assumption is vacuous. Since $\kappa\ge2$ for all
$\alpha\in(1/2,1)$, \eqref{ima.eq1} is implied by
\eqref{e:moment.eq}. 
}
\end{remark}

To state our main result, we need to introduce several key
quantities. Let
\begin{equation} \label{e:beta}
\beta=\frac{4-4\alpha}{3-2\alpha}\in (0,1)
\end{equation}
and 
\begin{equation}
\label{lpm.eq.defH}H=3/2-\alpha\in (1/2,1). 
\end{equation}
We denote by $(B_H(t):t\ge0)$ the standard fractional Brownian motion
with Hurst index $H$,  i.e. a zero mean Gaussian process with
continuous paths and covariance function 
\begin{equation}
\label{lpm.fbm}
E\left(B_H(s)B_H(t)\right)=\frac12\left(s^{2H}+t^{2H}-|s-t|^{2H}\right),
\, s,t\ge0\,.
\end{equation}

If $T_0$ is a standard exponential random variable independent of the
fractional Brownian motion, then 
\begin{equation} \label{e:tau}
\tau_\vep=\inf\left\{t\ge0:\, B_H(t)\le(2C_\alpha )^{-1/2}\vep
  t^{2H}-(C_\alpha /2)^{1/2} \sigma_Z^2\vep^{-1} T_0\right\}, \, \vep>0,
\end{equation}
is an a.s. finite and strictly positive random variable. Here  $\sigma_Z^2$ is
the variance of the noise in \eqref{e:noise.v} and 
\begin{equation}
\label{lpm.defC}C_\alpha=\frac{\Beta(1-\alpha,2\alpha-1)}{(1-\alpha)(3-2\alpha)}, 
\end{equation}
with $\Beta(\cdot,\cdot)$ the standard Beta function. 

We are now in a position to state the main result of this paper.
\begin{theorem}\label{lpm.t1}
Assume the finite exponential moment condition \eqref{eq.exp}, that
the coefficients satisfy the power-type condition \eqref{e:power.a}, the
regularity condition \eqref{eq.chf} and the moment equality condition
\eqref{e:moment.eq}.  Then for every $\vep>0$ the first non-occurrence
times \eqref{e:I.n} satisfy 
\begin{equation} \label{e:main.conv}
P\left(n^{-\beta}I_n(\vep)\in\cdot\bigr|E_0(n,\vep)\right)\Rightarrow
P\left(\tau_\vep\in\cdot\right), \, n\to\infty\,.
\end{equation}
\end{theorem}

\begin{remark} {\rm
It is worthwhile to observe that the limit law obtained in Theorem
\ref{lpm.t1} depends on the noise distribution only through its
variance $\sigma_Z^2$. This can be understood by noticing that in the
long memory case considered in this paper we have
$\var(X_1+\cdots+X_n)\gg n$; see Lemma \ref{lpm.l0} below. Therefore, the
events $E_j(n,\vep)$ should be viewed as moderate deviation events,
not large deviation events. It has been observed in many
situations that moderate deviation events are influenced by the
Gaussian  weak limit of the quantities of
interest. At the intuitive level, this explains why it is the variance
of the process that appears in the limit.

For comparison, in the short memory case \eqref{e:short.m}, we have
$\var(X_1+\cdots+X_n)\sim cn$ for some $c>0$, the events $E_j(n,\vep)$
should be viewed as  large deviation events, and their behaviour
depends on much more than just the variance of the noise. See
\cite{chakrabarty:samorodnitsky:2022} for details. 
}
\end{remark}

We start on the road to proving Theorem \ref{lpm.t1} by establishing
certain basic estimates that will be used throughout the paper. Denote 
\begin{equation} \label{e:A.j}
A_j=\sum_{i=0}^ja_i, \, j\in\bbz\,,
\end{equation}
with the convention that a sum (or an integral) is zero if the lower
limit exceeds the upper limit (so that $A_j=0$ for $j\le-1$, for
example).  Let 
\begin{equation} \label{e:S.n}
S_n=\sum_{i=0}^{n-1}X_i, \, n\ge1\,,
\end{equation}
and denote 
\begin{equation} \label{e:var.Sn}
\sigma_n^2=\Var(S_n), \, n\ge1\,.
\end{equation}
In the sequel we use the following notation. We will denote by 
\begin{equation}\label{e:varphiz}
\varphi_Z(t)=\log\left(\int_\bbr e^{tz}\,
  F_Z(dz)\right),\ t\in\bbr
\end{equation}
the log-Laplace transform of a noise variable. We will frequently use
the obvious facts 
\begin{equation}
\label{lpm.l3.eq1}\varphi \ \text{is convex and} \ \varphi_Z(x)\sim
x^2\sigma_Z^2/2, \  x\to0, 
\end{equation}
and
\begin{equation} \label{e:phiprime}
\varphi_Z^\prime \ \text{is continuous, nondecreasing and} \
\varphi_Z^\prime(x)= x \sigma_Z^2+ O(x^2), \  x\to0. 
\end{equation}
We will write 
$G_\theta$ for the probability
measure obtained by exponentially tilting the law $F_Z$ by
$\theta\in\bbr$. That is,
\begin{equation} \label{e:G.theta}
  G_\theta(dz)= \bigl(Ee^{\theta Z}\bigr)^{-1} e^{\theta z} F_Z(dz).
\end{equation}
It is clear that, as $\theta\to 0$, 
\begin{align} \label{e:tilt.facts}
&\int_\bbr z\, G_\theta(dz)\sim \theta\sigma_Z^2, \ 
\left| \int_\bbr z\, G_\theta(dz)-\theta\sigma_Z^2\right| =
  O(\theta^2) \  \text{and} \ = O(|\theta|^3) \ \ \text{if} \ \
                                     \kappa\geq 3,
\\
\notag &
 \int_\bbr |z|^k\,
G_\theta(dz)\to \int_\bbr |z|^k\, F(dz), \, k=1,2,\ldots.
\end{align}

\begin{lemma}\label{lpm.l0}
Asymptotically we have 
\begin{equation}
\label{lpm.l0.eq1}
A_j\sim(1-\alpha)^{-1}j^{1-\alpha}, \, j\to\infty
\end{equation}
and 
\begin{equation}
\label{lpm.l0.eq2}\sigma_n^2\sim C_\alpha \sigma_Z^2n^{3-2\alpha}, \, n\to\infty\,.
\end{equation}
Furthermore,  for any  $t>0$, as $n\to\infty$,
\begin{equation}
\label{lpm.l0.eq3}\sum_{i=0}^{[n^\beta t]}(A_i-A_{i-n})^2\sim K_1t^{3-2\alpha}n^{4-4\alpha}\,,
\end{equation}
and
\begin{equation}
  \label{lpm.l0.eq3.1}\sum_{i=n-[n^\beta t]+1}^n(A_i-A_{i-n})^2
  \sim\sum_{i=n+1}^{n+[n^\beta t]}(A_i-A_{i-n})^2 
\sim(1-\alpha)^{-2}n^{2-2\alpha+\beta}t\,,
\end{equation}
with 
\begin{equation}
\label{lpm.eq.defK1}K_1=(1-\alpha)^{-2}(3-2\alpha)^{-1}\,.
\end{equation}

Finally, for any $t>0$, as $n\to\infty$, 
\begin{equation}\label{lpm.l0.eq4}
\frac{\sigma_Z^{2}}{\sigma_n^{2}}\sum_{i=0}^\infty\left(A_i-A_{i-n}\right)\left(A_{i+[n^\beta
    t]}-A_{i+[n^\beta t]-n}\right)= 1-
n^{1-2\alpha}t^{3-2\alpha}(1+o(1))\,. 
\end{equation}
\end{lemma}

\begin{proof}
The claim \eqref{lpm.l0.eq1} is, of course, an immediate consequence
of the assumption \eqref{e:power.a}. For \eqref{lpm.l0.eq2}, first
note that 
\begin{align*}
R_n=\Cov(X_0,X_n)& 
\sim\sigma_Z^2\sum_{j=1}^\infty j^{-\alpha}(j+n)^{-\alpha}\\
\nonumber&\sim n^{1-2\alpha}\sigma_Z^2\int_0^\infty
           x^{-\alpha}(1+x)^{-\alpha}\, dx \\
   &=C_\alpha\sigma_Z^2(1-\alpha)(3-2\alpha)n^{1-2\alpha} 
\end{align*}
as $n\to\infty$. Therefore,
\begin{align*}
\sigma_n^2&=\sum_{i=-(n-1)}^{n-1}(n-|i|)R_{|i|}
          \sim2C_\alpha\sigma_Z^2(1-\alpha)(3-2\alpha)\sum_{i=0}^{n-1}(n-i)i^{1-2\alpha}\\
  \nonumber&\sim2C_\alpha \sigma_Z^2(1-\alpha)(3-2\alpha)n^{3-2\alpha}
             \int_0^1(1-x)x^{1-2\alpha}\, dx  
=C_\alpha \sigma_Z^2 n^{3-2\alpha}\,,
\end{align*}
which is \eqref{lpm.l0.eq2}. Next, for a fixed $t>0$ and  large $n$,
by \eqref{lpm.l0.eq1} and the fact that $\beta<1$, 
\begin{align*}
&\sum_{i=0}^{[n^\beta t]}(A_i-A_{i-n})^2=\sum_{i=0}^{[n^\beta t]}A_i^2
 \sim (1-\alpha)^{-2} \sum_{i=1}^{[n^\beta t]} i^{2-2\alpha}
\sim K_1\left(n^\beta t\right)^{3-2\alpha}, 
\end{align*}
proving \eqref{lpm.l0.eq3}. Similarly, 
\begin{align*}
\sum_{i=n-[n^\beta t]+1}^n\left(A_i-A_{i-n}\right)^2 
  \sim \sum_{i=n-[n^\beta t]+1}^nA_n^2
  \sim(1-\alpha)^{-2}n^{\beta+2-2\alpha}t\,,
\end{align*}
showing the first equivalence in \eqref{lpm.l0.eq3.1}
and the second equivalence can be shown in the same way.

 For \eqref{lpm.l0.eq4},  we start by writing 
\begin{equation}\label{lpm.l0.eq8}
S_n=\sum_{j=0}^\infty(A_j-A_{j-n})Z_{n-1-j}\,,n\ge1\,,
\end{equation}
so that 
\begin{align}\label{lpm.l0.sigman}
\sigma_n^2=\sigma_Z^2\sum_{j=0}^\infty(A_j-A_{j-n})^2\,,n\ge1\,.
\end{align}
Therefore, for large $n$, 
\begin{align}
\nonumber&\frac{\sigma_n^2}{\sigma_Z^2}-\sum_{i=0}^\infty(A_i-A_{i-n})(A_{i+[n^\beta t]}-A_{i+[n^\beta t]-n})\\
\nonumber&=\frac12\left[\sum_{i=0}^{[n^\beta t]-1}(A_i-A_{i-n})^2+\sum_{i=0}^\infty\left(A_{i+[n^\beta t]}-A_{i+[n^\beta t]-n}-A_i+A_{i-n}\right)^2\right]\\
\label{lpm.l0.eq11}&=\frac12\Biggl[\sum_{i=0}^{n-1}\left(A_i-A_{i-[n^\beta t]}\right)^2\\
\nonumber&\,\,\,\,+\sum_{i=n-[n^\beta t]}^\infty\left(A_{i+[n^\beta t]}-A_{i+[n^\beta t]-n}-A_i+A_{i-n}\right)^2\Biggr]\,. 
\end{align}
By \eqref{lpm.l0.eq1}, 
\begin{align*}
&\sum_{i=0}^{n-1}\left(A_i-A_{i-[n^\beta t]}\right)^2
\sim(1-\alpha)^{-2}\sum_{i=1}^{n-1}\left(i^{1-\alpha}-(i-[n^\beta t])_+^{1-\alpha}\right)^2\\
 \nonumber&\sim n^{4-4\alpha}t^{3-2\alpha}(1-\alpha)^{-2}\int_0^\infty \left[y^{1-\alpha}-(y-1)_+^{1-\alpha}\right]^2dy
\end{align*}
as $n\to\infty$. By \eqref{e:pickard} with $H=3/2-\alpha$, 
\begin{align}\label{lpm.l0.eq12}
&\int_0^\infty \left[y^{1-\alpha}-(y-1)_+^{1-\alpha}\right]^2dy \\
\nonumber&=\left[\left(3-2\alpha\right)(1-\alpha)\right]^{-1}\frac{\sin(\pi\alpha)}\pi \Gamma(2\alpha-1)\Gamma(2-\alpha)^2\\
 \nonumber&=\frac{1-\alpha}{3-2\alpha}\Beta\left(2\alpha-1,1-\alpha\right) 
=(1-\alpha)^{2} C_\alpha, 
\end{align}
so 
\begin{equation}\label{lpm.l0.eq10}
\sum_{i=0}^{n-1}\left(A_i-A_{i-[n^\beta t]}\right)^2\sim 
C_\alpha t^{3-2\alpha}n^{4-4\alpha}, \, n\to\infty\,.
\end{equation}
Since
\begin{align} \label{e:interm.10}
\sum_{i=n}^{\infty}\left(A_i-A_{i-[n^\beta t]}\right)^2
  =O\left( n^{2\beta}\sum_{i=n}^{\infty} i^{-2\alpha}\right)
  =O\bigl( n^{2\beta+1-2\alpha}\bigr) = o\bigl( n^{4-4\alpha}\bigr), 
\end{align}
we conclude also that 
\begin{equation}\label{full.s.eq10}
\sum_{i=0}^{\infty}\left(A_i-A_{i-[n^\beta t]}\right)^2\sim 
C_\alpha t^{3-2\alpha}n^{4-4\alpha}, \, n\to\infty\,.
\end{equation}
It follows from \eqref{e:interm.10} and \eqref{full.s.eq10} that 
\begin{align*}
&\sum_{i=n-[n^\beta t]}^\infty\left(A_{i+[n^\beta t]}-A_{i+[n^\beta t]-n}-A_i+A_{i-n}\right)^2\\
&=\sum_{j=0}^\infty\left[-A_j+A_{j-[n^\beta  t]}+\left(A_{j+n}-A_{j+n-[n^\beta t]}\right)\right]^2
  \sim C_\alpha t^{3-2\alpha}n^{4-4\alpha}\,.
\end{align*}
 In combination with  \eqref{lpm.l0.eq11} and \eqref{lpm.l0.eq10} we
 obtain 
\[
\frac{\sigma_n^2}{\sigma_Z^2}-\sum_{i=0}^\infty(A_i-A_{i-n})(A_{i+[n^\beta
  t]}A_{i+[n^\beta t]-n})\sim C_\alpha t^{3-2\alpha}n^{4-4\alpha}\,.
\]
Dividing both sides by $\sigma_Z^{-2}\sigma_n^2$ and appealing to \eqref{lpm.l0.eq2}, \eqref{lpm.l0.eq4} follows. 
\end{proof}

We now consider certain large deviations of the partial sum $S_n$ under a change
of measure. With an eye towards a subsequent application, we allow the
partial sum, given in the form \eqref{lpm.l0.eq8}, to be
``corrupted''. For  $n\ge1$ and $t\ge0$ we define 
\begin{align}
\label{lpm.eq0}\xi_n^1(t) &=\sum_{i=1}^{[n^\beta t]}\left(A_i-A_{i-n}\right)Z_{n-i-1}\,,\\
\label{e:middle.pr} \xi_n^2(t) &=  \sum_{i=n-[n^\beta t]}^{n-1} \left(A_i-A_{i-n}\right)Z_{n-i-1}\,,\\
\label{lpm.eq1}\xi_n^3(t) &=  \sum_{i=n+1}^{n + [n^\beta t]} \left(A_i-A_{i-n}\right)Z_{n-i-1}\,.
\end{align}

\begin{lemma}\label{lpm.l1}
Fix $t_1,t_2,t_3>0$ and denote 
\begin{equation} \label{e:Sn.corr}
\bar S_n=S_n-\sum_{i=1}^3\xi_n^i(t_i), \, n\ge1\,.
\end{equation}
Let $(\gamma_n)$, $(\theta_n)$ and $(\eta_n)$ be real sequences satisfying
\[
\gamma_n=o\left(n^{3/2-\alpha}\right), \ 
\theta_n=o\left(n^{-(1-\alpha)}\right), \ 
1\ll\eta_n\ll n^{1/2}\,.
\]
If $\tilde S_n$ is a random variable with the law 
\begin{equation}
\label{lpm.l1.eq1}P\left(\tilde S_n\in
  dx\right)=\left(E(e^{\theta_n\bar
    S_n})\right)^{-1}e^{\theta_nx}P\left(\bar S_n\in dx\right),\,
n\ge1\,, 
\end{equation}
then for all $x\in\bbr$ and $h>0$,
\begin{equation}
\label{lpm.l1.eq2}P\left(\eta_n\sigma_n^{-1}\left(\tilde S_n-E(\tilde
    S_n)+\gamma_n\right)\in[x,x+h]\right)\sim\eta_n^{-1}(2\pi)^{-1/2}h,\,
n\to\infty. 
\end{equation}
Furthermore, 
\begin{equation}
\label{lpm.l1.eq3} \sup_{n\ge 1, \, x\in\bbr}\eta_nP\left(\eta_n\sigma_n^{-1}\tilde S_n\in[x,x+1]\right)<\infty\,.
\end{equation}
\end{lemma}

\begin{proof}
Let $(\tilde Z_{ni}, \, n\ge1,\, i\ge0)$ be a collection of
independent random variables such that the law of $\tilde Z_{ni}$ is
$G_{(A_i-A_{i-n})\theta_n}$ in the notation of \eqref{e:G.theta}. 
 Then for large $n$, 
\begin{equation}
\label{lpm.l1.eq4} \tilde S_n\eid A_0\tilde Z_{n0}+(A_n-A_0)\tilde Z_{nn}+ \sum_{i=[n^\beta t_1]+1}^{n-[n^\beta t_2]-1}A_i\tilde Z_{ni}+\sum_{i=n+[n^\beta t_3]+1}^\infty(A_i-A_{i-n})\tilde Z_{ni}\,.
\end{equation}

The proof applies to \eqref{lpm.l1.eq4} the bound 
\eqref{e:BE.nonid} in the appendix, with $n=\infty$. For any $z\in\bbr$
\begin{align}\label{lpm.l1.eq5}
& \left|P\left(\tilde S_n-E(\tilde S_n)\le z\sqrt{\Var(\tilde S_n)}\right)-\Phi(z)\right| \\
  & \le C_u\left(\Var(\tilde   S_n)\right)^{-3/2}
    \sum_{i=0}^\infty|A_i-A_{i-n}|^3
    E\left(|\tilde  Z_{ni}-E\tilde  Z_{ni}|^3\right),   \,   n\ge1 \,. \nonumber
\end{align}
 It is immediate from \eqref{e:power.a} that 
 \begin{equation} \label{e:A.incr}
\sup_{i\ge0}\left|A_i-A_{i-n}\right|=O(n^{1-\alpha})\,, 
\end{equation}
so that 
$$
\lim_{n\to\infty}\theta_n\sup_{i\ge0}\left|A_i-A_{i-n}\right|=0\,.
$$
It follows from \eqref{e:tilt.facts} that 
\begin{equation}
\label{lpm.l1.eq6} 
 E\tilde Z_{ni}\to 0, \ \  \var(\tilde Z_{ni})\to \sigma_Z^2, \ \ 
E\left(|\tilde Z_{ni}-E\tilde
   Z_{ni}|^3\right)\to \int_{-\infty}^\infty |z^3|  \, F_Z(dz)
\end{equation}
uniformly in $i$ as $n\to\infty$.  Since it is an elementary
conclusion from Lemma \ref{lpm.l0} that for any $\kappa>1/\alpha$, 
\begin{align}\label{lpm.l1.eqcube}
\sum_{i=0}^\infty|A_i-A_{i-n}|^\kappa=O\bigl(n^{\kappa+1-\kappa\alpha}\bigr), 
\end{align}
it follows from \eqref{lpm.l1.eq5} that 
\begin{align*}
&\sup_{z\in\bbr}\left|P\left(\tilde S_n-E(\tilde S_n)\le z\sqrt{\Var(\tilde S_n)}\right)-\Phi(z)\right|\\
&=O\left(n^{4-3\alpha}\left(\Var(\tilde S_n)\right)^{-3/2}\right)\,.
\end{align*}
Using \eqref{lpm.l1.eq6} again we see that 
\begin{align} \label{e:var.tilde}
\Var(\tilde
  S_n)&\sim\sigma_n^2-\sum_{i=1}^3\Var\left(\xi_n^i(t_i)\right) \sim
        C_\alpha \sigma^2_Zn^{3-2\alpha}, 
\end{align}
with the second equivalence following from various claims in Lemma \ref{lpm.l0}.  Thus,
\begin{equation}
\label{lpm.l1.eq7}\sup_{z\in\bbr}\left|P\left(\tilde S_n-E(\tilde S_n)\le z\sqrt{\Var(\tilde S_n)}\right)-\Phi(z)\right|=O(n^{-1/2})=o\left(\eta_n^{-1}\right)\,.
\end{equation}

Therefore, for  $x\in\bbr$ and $h>0$, as $n\to\infty$,
\begin{align*}
&P\left(\eta_n\sigma_n^{-1}\left(\tilde S_n-E(\tilde S_n)+\gamma_n\right)\in[x,x+h]\right)\\
&                 =o\left(\eta_n^{-1}\right)
  +\int_\bbr \one\bigl[ 
\Var(\tilde S_n)^{-1/2}(x\eta_n^{-1}\sigma_n-\gamma_n)\leq z\\
&\hspace{1in}  \leq
      \Var(\tilde S_n)^{-1/2}((x+h)\eta_n^{-1}\sigma_n-\gamma_n)\bigr]\phi(z)\, dz,
\end{align*}
where $\phi$ is the standard normal density.  The assumptions on
$\gamma_n$ and  $\eta_n$ along with \eqref{e:var.tilde}
imply that the integration interval shrinks towards the origin. Thus,
the integral above is asymptotically equivalent to 
$\eta_n^{-1}\phi(0)h$, and  \eqref{lpm.l1.eq2} follows.  Boundedness
of $\phi$ in the above integral establishes \eqref{lpm.l1.eq3}.
\end{proof}

We now look more closely at the processes defined in \eqref{lpm.eq0},
\eqref{e:middle.pr} and \eqref{lpm.eq1}. The next lemma describes the
limiting distribution of their increments under the same change of measure as in the
previous lemma. 

\begin{lemma}\label{lpm.l2} Suppose that $\theta_n\in\bbr$ satisfies $
\theta_n=o\left(n^{-(1-\alpha)}\right)$. 
Fix $0\le s<t$ and consider random variables with the laws 
\[
P(U_{ni}\in dx)=c_{ni}e^{\theta_nx}P\left(\xi_n^i(t)-\xi_n^i(s)\in
  dx\right), \, i=1,2,3, \, n\ge1\,,
\]
with appropriate $c_{ni}$.  Then, as $n\to\infty$,
\begin{equation}
\label{lpm.l2.eq1}n^{-(2-2\alpha)}\left(U_{n1}-E (U_{n1})\right)\Rightarrow N\bigl(0, K_1\sigma_Z^2\left(t^{3-2\alpha}-s^{3-2\alpha}\right)\bigr)\,,
\end{equation}
where $K_1$ is given in \eqref{lpm.eq.defK1},  and for $i=2,3$, 
\begin{equation}
\label{lpm.l2.eq2}n^{-(1-\alpha+\beta/2)}\left(U_{ni}-E (U_{ni})\right)\Rightarrow N \bigl(0,  (1-\alpha)^{-2}\sigma_Z^2(t-s)\bigr)\,.
\end{equation}
\end{lemma}

\begin{proof} For large $n$, 
\[
U_{n1}\eid \sum_{i=[n^\beta s]+1}^{[n^\beta t]} A_i\tilde Z_{ni}
\]
with $(\tilde Z_{ni})$  as in  the previous lemma. That is,
$U_{n1}-E(U_{n1})$ is the sum of
independent zero mean random variables. By \eqref{lpm.l1.eq6} and
\eqref{lpm.l0.eq3}, 
\[
\Var(U_{n1})\sim\sigma_Z^2\sum_{i=[n^\beta s]+1}^{[n^\beta t]}A_i^2\sim K_1\sigma_Z^2 n^{4-4\alpha}\left(t^{3-2\alpha}-s^{3-2\alpha}\right)\,,
\]
and a similar calculation using the third moment bound in \eqref{lpm.l1.eq6} 
verifies the Lindeberg conditions of the central limit theorem. Hence
\eqref{lpm.l2.eq1} follows, and the calculations for \eqref{lpm.l2.eq2} are similar. 
\end{proof}

Consider the overshoot defined by
\begin{equation}\label{lpm.eq.defTn}
T_n^*=S_n-n\vep, \, n\ge1\,.
\end{equation}
Conditionally on the event $E_0=E_0(n,\vep)$ in \eqref{e:E.j}
the overshoot is nonnegative. 
The next lemma is a joint weak limit theorem for the joint law of the overshoot
and the 
processes defined in \eqref{lpm.eq0}, 
\eqref{e:middle.pr} and \eqref{lpm.eq1}. The joint law is computed
conditionally on $E_0$. 

\begin{lemma}\label{lpm.l3}
Let 
\begin{equation}
\label{eq.defzetan}\zeta_n= n\vep/\sigma_n^2, \ n\ge1\,,
\end{equation}
 
Conditionally on $E_0$, as $n\to\infty$,
\begin{align*}
&\Biggl(\zeta_nT_n^*, \,  \left( n^{2\alpha-2}
                 \left(\xi_n^1(t)-\sum_{i=1}^{[n^\beta
                 t]}A_i\int_{-\infty}^\infty z\, G_{\zeta_n A_i}(dz)\right), \,  t \ge 0 \right), \\
&\,\,\Bigl(n^{\alpha-\beta/2-1}   \left( \xi_n^{ 2 }(t) -  \sum_{i=n-[n^\beta t]}^{n - 1}   A_i
      \int_{-\infty}^\infty z\, G_{\zeta_n A_i}(dz)\right), \, t\ge0\Bigr),\\
&\,\,\Bigl(n^{\alpha-\beta/2-1} 
\left( \xi_n^3(t) -  \sum_{i=n+1}^{n + [n^\beta t]} (A_i-A_{i-n})\int_{-\infty}^\infty z\, G_{\zeta_n (A_i-A_{i-n})}(dz)\right), \, t\ge0\Bigr)  \Biggr)
\end{align*}
\begin{align*}
\Rightarrow& \Bigl(T_0,\, \bigl(K_1^{1/2}\sigma_Z B_1(t^{3-2\alpha}), \, t\ge0\bigr),\\
&\,\,\,\,\,\,\bigl((1-\alpha)^{-1}\sigma_Z B_2(t), \, t\ge0\bigr),\bigl((1-\alpha)^{-1}\sigma_Z B_3(t), \, t\ge0\bigr)\Bigr)\,,
\end{align*}
in finite dimensional distributions, where $T_0$ is a standard
exponential random variable independent of independent standard
Brownian motions $B_1$, $B_2$,  and $B_3$, $K_1$ is the constant in \eqref{lpm.eq.defK1}
and $G_\theta$ is the exponentially tilted law in    \eqref{e:G.theta}. 
\end{lemma}

\begin{proof}
Denote 
\begin{equation}\label{eq.defpsin}
\psi_n(s)=\frac{\sigma_n^2}{n^2}\log E\left[\exp\left(s\frac{n}{\sigma_n^2}S_n\right)\right]=\frac{\sigma_n^2}{n^2}\sum_{j=0}^\infty \varphi_Z\left(\sigma_n^{-2}n(A_j-A_{j-n})s\right)\,,
\end{equation}
where the second equality follows from \eqref{lpm.l0.eq8}. By
\eqref{lpm.l3.eq1}, \eqref{lpm.l0.eq2} and \eqref{e:A.incr} we see
that 
\begin{equation} \label{e:psi.lim}
\lim_{n\to\infty} \psi_n(s)=s^2/2
\end{equation}
uniformly for $s$ in a compact set. Furthermore, the sum in
\eqref{eq.defpsin} can be differentiated term by term, and it follows
by \eqref{e:phiprime}, \eqref{lpm.l0.eq2} and \eqref{e:A.incr} that 
\begin{equation}
\label{lpm.l3.eq2}
\lim_{n\to\infty}\psi_n'(s) = s, 
\end{equation}
also uniformly on compact sets.  Since $\psi_n'$ is increasing and
continuous, for large $n$ there exists a unique $\tau_n>0$ such that 
\begin{equation}\label{eq.deftaun}
\psi_n'(\tau_n)=\vep\,.
\end{equation}
It is immediate that $\tau_n\to\vep$ as $n\to\infty$.  Denoting
\begin{equation}
\label{eq.defthetan}\theta_n=\sigma_n^{-2}n\tau_n, \, n\ge1\,,
\end{equation}
we have 
\begin{align}\label{lpm.l3.eq3}
\left( E\left(e^{\theta_nS_n}\right)
       \right)^{-1}E\left(S_ne^{\theta_nS_n}\right) = n\vep.
\end{align}

Fix $k\ge1$ and for each $i=1,2,3$ fix points 
$0=t_{i0}< t_{i1}<\ldots<t_{ik}$.
Denote 
\[
\bar S_n=S_n-\sum_{i=1}^3\xi_n^i(t_{ik}), \, n\ge1\,.
\]

Let $U_{nij},\, n\ge1,\,  i=1,2,3,\,  j=1,\ldots,k, \, \tilde S_n, \,
n\geq 1$ be 
independent random variables, with 
\begin{align*}
&P\left(U_{nij}\in dx\right) \\
&=\left(E\left(e^{\theta_n(\xi^i_n(t_{ij})-\xi^i_n(t_{i\,j-1}))}\right)\right)^{-1}
                 e^{\theta_nx}
                 P\left(\xi^i_n(t_{ij})-\xi^i_n(t_{i\,j-1})\in
                 dx\right)\,,
\end{align*}
and
\begin{align*}
 &P\left( \tilde S_n\in dx\right) = \left(E\left(e^{\theta_n\bar
   S_n}\right)\right)^{-1} e^{\theta_nx}    P\left( \bar S_n\in dx\right) 
   \end{align*}
for $n\ge1$, $i=1,2,3$ and $j=1,\ldots,k$.  Let
\begin{equation} \label{e:mus}
\mu_{nij}=E\left(U_{nij}\right), \ \mu_n=E(\tilde S_n).
\end{equation}
It follows from  \eqref{lpm.l3.eq3} that 
\begin{equation}
\label{lpm.l3.eq4}\mu_n+\sum_{i=1}^3\sum_{j=1}^k\mu_{nij}=n\vep, \, n\ge1\,.
\end{equation}

Let $t>0$ and  $(\alpha_{ij})\subset\bbr$. We have
\begin{align*}
& P\biggl(\bigl\{  T_n^*>t\sigma_n^2/n\vep
                 \bigr\}\cap\Bigl(\bigcap_{j=1}^k\bigl\{ n^{2\alpha-2}\left(\xi^1_n(t_{1j})-\xi^1_n(t_{1\,j-1}) -\mu_{n1j}\right)>\alpha_{1j} \bigr\}\Bigr)\\
&\cap\Bigl(\bigcap_{2\le i\le3,\, 1\le j\le k}\bigl\{n^{\alpha-\beta/2-1}\left(\xi_n^i(t_{ij})-\xi_n^i(t_{i\,j-1})-\mu_{nij}\right)>\alpha_{ij}\bigr\}\Bigr)\biggr)
\end{align*}
\begin{align*}
&=\int_{\bbr^{3k+1}}                                                                                                                 \one\biggl(x>n\vep+t\sigma_n^2/n\vep-\sum_{i=1}^3\sum_{j=1}^ks_{ij}\biggr)\\
&\hskip 0.65in\one\left(s_{1j}>n^{2-2\alpha}\alpha_{1j}+\mu_{n1j}\,,1\le j\le k\right)\\
&\hskip 0.65in  \one\left(s_{ij}>n^{1-\alpha+\beta/2}\alpha_{ij}+\mu_{nij}\,,i=2,3\,,j=1,\ldots, k\right) \\
&\hskip 1.5in P(\bar S_n\in dx)\prod_{i=1}^3\prod_{j=1}^k P\bigl(\xi^i_n(t_{ij}) - \xi^i_n(t_{i\,j-1}) \in ds_{ij}\bigr)
\end{align*}
\begin{align*}
  &=\int_{\bbr^{3k+1}}
    \one\biggl(x>n\vep+t\sigma_n^2/n\vep-\sum_{i=1}^3\sum_{j=1}^ks_{ij}\biggr)\\
  &\hskip 0.65in\one\left(s_{1j}>n^{2-2\alpha}\alpha_{1j}+\mu_{n1j}\,,1\le j\le k\right)\\
& \hskip 0.65 in\one\left(s_{ij}>n^{1-\alpha+\beta/2}\alpha_{ij}+\mu_{nij}\,,i=2,3\,,1\le j\le k\right)
 \\
& \hskip 1.5in
      \exp\biggl(-\theta_nx-\theta_n\sum_{i=1}^3\sum_{j=1}^ks_{ij}\biggr)P\left(\tilde
      S_n\in dx\right)\\
      &\hskip1.5in E\left(e^{\theta_nS_n}\right)\prod_{i=1}^3\prod_{j=1}^kP(U_{nij}\in
      ds_{ij})
\end{align*}
\begin{align*}
 &=c_n\int_{\bbr^{3k}}\one\bigl(\min_{i,j}(u_{ij}-\alpha_{ij})>0\bigr)
   \prod_{j=1}^kP\left(n^{2\alpha-2}\bigl(U_{n1j}-\mu_{n1j}\bigr)\in
   du_{1j}\right)  \\
  &\hskip 1.9in\prod_{i=2}^3\prod_{j=1}^k
   P\left(n^{\alpha-\beta/2-1}\bigl(U_{nij}-\mu_{nij}\bigr)\in
   du_{ij}\right) \\
& \hskip 0.3in \int_\bbr e^{-z} \one\bigl(z>t\theta_n \sigma_n^2/n\vep\bigr) 
P\Bigl(\theta_n\bigl(\tilde S_n-\mu_n +
                       \gamma_n(u_{11},\ldots,u_{3k})\bigr)\in dz\Bigr) \,,
\end{align*}
with
\begin{equation} \label{e:c.n}
  c_n=e^{-\theta_nn\vep} E\left(e^{\theta_nS_n}\right)
\end{equation}
and 
\[
\gamma_n(u_{11},\ldots,u_{3k})=n^{2-2\alpha}\sum_{j=1}^ku_{1j} +n^{1-\alpha+\beta/2}\sum_{i=2}^3\sum_{j=1}^ku_{ij}\,.
\]
Let $\theta_n$ be as above and $\eta_n=\sigma_n\theta_n$. 
For $n\ge1$,  we introduce the notation 
\begin{align*}
 & f_n(u_{11},\ldots,u_{3k})\\
  =&\eta_n\int_0^\infty   e^{-z}
  \one\bigl( z> t\theta_n \sigma_n^2/n\vep\bigr) 
P\Bigl(\theta_n\bigl(\tilde S_n-\mu_n + \gamma_n(u_{11},\ldots,u_{3k})\bigr)\in
                 dz\Bigr)\, .
   \end{align*}
 
Fix $(u_{ij})$ and let $u^{(n)}_{ij}\to u_{ij}$ as $n\to\infty$ for
all $i,j$.  Let us denote
$\gamma_n=\gamma_n\bigl(u_{11}^{(n)},\ldots,u_{3k}^{(n)}\bigr)$.
With $\theta_n$  and $\eta_n$ already defined, 
we use  Lemma \ref{lpm.l1} with this $\gamma_n$. It is
elementary to check that the hypothesis of the lemma  are satisfied.
Since $t\theta_n \sigma_n^2/n\vep\to t$, it follows
from  \eqref{lpm.l1.eq2}  that for all fixed  $T>t$,
\begin{align*}
&  \int_\bbr e^{-z} \one\bigl(t\theta_n \sigma_n^2/n\vep<z\leq T\bigr) 
P\Bigl(\theta_n\bigl(\tilde S_n-\mu_n +
                       \gamma_n\bigr)\in
                 dz\Bigr)\\
                 &   \sim\eta_n^{-1}(2\pi)^{-1/2}\int_t^Te^{-z}\, dz,
\end{align*}
and if follows from  \eqref{lpm.l1.eq3} that 
\begin{equation*}
  \lim_{T\to\infty}\limsup_{n\to\infty}\eta_n \int_\bbr e^{-z}
  \one\bigl( z> T\bigr) 
P\Bigl(\theta_n\bigl(\tilde S_n-\mu_n + \gamma_n\bigr)\in
                 dz\Bigr) =0\,,
\end{equation*}
showing that
\[
\lim_{n\to\infty}f_n\left(u^{(n)}_{11},\ldots,u^{(n)}_{3k}\right)=(2\pi)^{-1/2}e^{-t}\,.
\]
Another application of \eqref{lpm.l1.eq3} implies that
\[
 \sup_{\{u_{ij}\}\subset\bbr}f_n(u_{11},\ldots,u_{3k})<\infty\,.
\]
It follows immediately from Lemma \ref{lpm.l2} and bounded convergence
theorem that
\begin{align} \label{e:int.conv.gen}
&E\Bigl[ f\bigl( n^{2\alpha-2}(U_{n11}-\mu_{n11}), \ldots,
n^{\alpha-\beta/2-1}(U_{n3k}-\mu_{n3k} )   \bigr) \\
\notag  & \one\Bigl( n^{2\alpha-2}(U_{n1j}-\mu_{n1j})>\alpha_{1j},  \ n^{\alpha-\beta/2-1}(U_{nij}-\mu_{nij} )
>\alpha_{ij}, \, i=2,3, \\ 
\notag &\hskip3in j=1,\ldots, k,\Bigr)   \Bigr]
\\
\notag \to& 
(2\pi)^{-1/2}P\left(T_0>t\,,G_{ij}>\alpha_{ij}\text{ for all }i,j\right)\,,
\end{align}
with $T_0$ standard exponential and $(G_{ij}:i=1,2,3,\,j=1,\ldots,k)$
independent zero mean Gaussian random variables, independent of $T_0$,
with 
\[
\Var(G_{1j})=K_1\sigma_Z^2\left(t_{1j}^{3-2\alpha}-t_{1\,j-1}^{3-2\alpha}\right),\,
1\le j\le k\,,
\]
and for $i=2,3$,
\[
\Var(G_{ij})=(1-\alpha)^{-2}\sigma_Z^2(t_{ij}-t_{i\,,j-1}), \, 1\le j\le k\,.
\]
A simple way to verify the convergence above is to appeal to the
Skorohod representation and replace the weak convergence in Lemma
\ref{lpm.l2} by the a.s. convergence.

Notice that using \eqref{e:int.conv.gen} with 
$t=0$ and $\alpha_{ij}=-\infty$ for all $i,j$ tells us that
\begin{equation} \label{e:P.E0}
P(E_0)\sim (2\pi)^{-1/2}c_n/\eta_n
= (2\pi)^{-1/2}e^{-\theta_nn\vep}
  E\left(e^{\theta_nS_n}\right)/(\sigma_n\theta_n). 
\end{equation}
Dividing \eqref{e:int.conv.gen} by \eqref{e:P.E0} gives us the
statement of the lemma apart from a possibly different centring.  In
order to complete the proof, it suffices to show that as $n\to\infty$,
for $j=1,\ldots,k$, 
\begin{align}
\label{lpm.l3.eq6}
\mu_{n1j}&=\sum_{i=[n^\beta t_{1 j-1}]+1}^{[n^\beta
           t_{1j}]}A_i\int_{-\infty}^\infty z\, G_{\zeta_n A_i}(dz)+o\left(n^{2-2\alpha}\right)\,,\\
\label{lpm.l3.eq7}
\mu_{n2j}&= \sum_{i=n-[n^\beta t_{nj}]}^{n -[n^\beta t_{n j-1}] }
           A_i\int_{-\infty}^\infty z\, G_{\zeta_n A_i}(dz)
           +o\left(n^{1+\beta/2-\alpha}\right)\,,\\
\label{lpm.l3.eq9}
\mu_{n3j}&= \sum_{i=n+[n^\beta t_{n j-1}]}^{n + [n^\beta t_{nj}]}
           (A_i-A_{i-n})\int_{-\infty}^\infty z\, G_{\zeta_n (A_i-A_{i-n})}(dz) +o\left(n^{1+\beta/2-\alpha}\right)\,.
\end{align}
For simplicity of notation we prove these statements for $j=1$. 
 For
$\theta_n$ as in \eqref{eq.defthetan}, let 
$(\tilde Z_{ni}, \, n\ge1,\, i\ge0)$ be a collection of
independent random variables such that the law of $\tilde Z_{ni}$ is
$G_{(A_i-A_{i-n})\theta_n}$. Since both $\theta_nA_i$ and $\zeta_nA_i$
  converge to zero uniformly in $i\leq n^\beta t_{11}$, we can use
  \eqref{e:tilt.facts} to write 
\begin{align*}
\mu_{n11}&=\sum_{i=1}^{[n^\beta t_{11}]}A_iE\left(\tilde Z_{ni}\right) 
=\sum_{i=1}^{[n^\beta t_{11}]}A_i\int_{-\infty}^\infty z\, G_{\theta_nA_i}(dz)\\
&=\sum_{i=1}^{[n^\beta t_{11}]}A_i\int_{-\infty}^\infty z\,
   G_{\zeta_nA_i}(dz)+o\left(\zeta_n\sum_{i=1}^{[n^\beta t_{11}]}A_i^2\right)\,. 
\end{align*}
It follows from \eqref{lpm.l0.eq2} and \eqref{lpm.l0.eq3} that 
\[
 \zeta_n\sum_{i=1}^{[n^\beta
  t_{11}]}A_i^2=o\left(n^{2-2\alpha}\right)\,, 
\]
and we obtain \eqref{lpm.l3.eq6}  (for $j=1$). 

For \eqref{lpm.l3.eq7} with $j=1$ we notice that by
\eqref{e:phiprime}, 
 \begin{equation}
\label{lpm.l3.eq10}E\left(\tilde Z_{ni}\right)=\theta_n(A_i-A_{i-n})\sigma_Z^2+O\left(\theta_n^2(A_i-A_{i-n})^2\right)\,,
\end{equation}
uniformly in $i\ge0$,  as $n\to\infty$.  Thus,
\begin{align*}\nonumber
\mu_{n21}& =\sigma_Z^2\sigma_n^{-2}n\tau_n\sum_{i=n-[n^\beta t_{21}]}^{n-1}A_i^2+O\left( \theta_n^2\sum_{i=n-[n^\beta t_{21}]}^{n-1}A_i^3 \right) \,.
\end{align*}
It follows from Lemma \ref{lpm.l0}  that 
\begin{align*}
\theta_n^2\sum_{i=n-[n^\beta t_{21}]}^{n-1}A_i^3&=O\left(n^{\alpha+\beta-1}\right)
= o\left(n^{1-\alpha+\beta/2}\right)\,.
\end{align*}
Therefore,
\begin{equation}\label{lpm.l3.eq8}
\mu_{n21}=\sigma_Z^2\sigma_n^{-2}n\tau_n\sum_{i=n-[n^\beta t_{21}]}^{n-1}A_i^2+o\left(n^{1-\alpha+\beta/2}\right)
\end{equation}
and, similarly, 
\[
\sum_{i=n-[n^\beta t_{21}]}^{n - 1} A_i\int_{-\infty}^\infty z\, G_{\zeta_n A_i}(dz) = \sigma_Z^2\zeta_n\sum_{i=n-[n^\beta t_{21}]}^{n-1}A_i^2+o\left(n^{1-\alpha+\beta/2}\right)\,.
\]
Another appeal to Lemma \ref{lpm.l0} shows that  for
\eqref{lpm.l3.eq7}  we only need to argue that 
\begin{equation}\label{lpm.l3.eq11}
\tau_n=\vep+o\left(n^{1-\alpha-\beta/2}\right)\,,n\to\infty\,.
\end{equation}
However, by \eqref{e:tilt.facts}, 
\[
\psi_n'(s)=s+O\left(n\sigma_n^{-4}\sum_{j=0}^\infty(A_j-A_{j-n})^3\right)\,,
\]
uniformly for $s$ in compact sets.  Using this and
\eqref{lpm.l1.eqcube}, we obtain 
\begin{align*}
\vep&=\psi_n'(\tau_n)\\
&=\tau_n+O\left(n\sigma_n^{-4}\sum_{j=0}^\infty(A_j-A_{j-n})^3\right)\\
&=\tau_n+O(n^{\alpha-1})
=\tau_n+o\left(n^{1-\alpha-\beta/2}\right)\,.
\end{align*}
 This establishes \eqref{lpm.l3.eq11} and, hence, \eqref{lpm.l3.eq7}
 with $j=1$.  The proof  of \eqref{lpm.l3.eq9} is similar. 
\end{proof}

None of the statements proved so far required the additional
assumptions stated at the beginning of this section. These assumptions
start to play a role now.

The next several lemmas require additional notation designed to focus
on the contribution of individual noise variables on $S_n$. For
$n\ge1$ and $i,j\ge0$, $i\neq j$, we set 
\[
  S_n'(i)=S_n-(A_i-A_{i-n})Z_{n-i-1}\,,
\]
\[
S_n'(i,j)=S_n-(A_i-A_{i-n})Z_{n-i-1}-(A_j-A_{j-n})Z_{n-j-1}\,,
\]
and, with $\zeta_n$ given by \eqref{eq.defzetan}, we 
 let $\hat S_n$, $\hat S_{ni}$, $\hat S_n(i,j)$ be random
 variables with distributions 
\[
P(\hat S_n\in ds)\propto e^{\zeta_ns}P(S_n\in ds)\,,
\]
\[
P(\hat S_n(i)\in ds)\propto e^{\zeta_ns}P(S_n'(i)\in ds)\,,
\]
\[
P(\hat S_n(i,j)\in ds)\propto e^{\zeta_ns}P(S_n'(i,j)\in ds)\,.
\]

Denote the characteristic functions of $\sigma_n^{-1}(\hat
S_n-n\vep)$, $\sigma_n^{-1}(\hat S_n(i)-n\vep)$ and
$\sigma_n^{-1}(\hat S_n(i,j)-n\vep)$ by $\phi_n$,  $\phi_{ni}$ and
$\phi_{nij}$,  respectively.  For $\mu\in\bbr$ and $\sigma\ge0$ we
denote by $\phi_G(\mu;\sigma^2;\cdot)$ the characteristic function of
$N(\mu,\sigma^2)$.

\begin{lemma}\label{lpm.newl0}
  Let $\kappa$ be given by \eqref{eq.kappa} and assume that
  \eqref{e:moment.eq} holds. Then 
the following statements hold uniformly in $t\in\bbr$:
\begin{equation}
\label{lpm.newl0.claim1}\left|\phi_n(t)-\phi_G(0;1;t)\right|=O\left(n^{1/2-\kappa(1-\alpha)}(1+|t|)^{\kappa+1}\right)\,,
\end{equation}
\begin{equation}
\label{lpm.newl0.claim2}\sup_{i\ge0}\left|\phi_{ni}(t)-\phi_G\left(\sigma_n^{-1}n\vep(\lambda_{ni}-1);\lambda_{ni};t\right)\right|=O\left(n^{1/2-\kappa(1-\alpha)}(1+|t|)^{\kappa+1}\right)\,,
\end{equation}
\begin{align}
\label{lpm.newl0.claim3}&\sup_{i,j\ge0 \atop i\neq
                          j}\left|\phi_{nij}(t)-\phi_G\left(\sigma_n^{-1}n\vep(\lambda_{nij}-1);\lambda_{nij};t\right)\right|\\
\notag &                          =O\left(n^{1/2-\kappa(1-\alpha)}(1+|t|)^{\kappa+1}\right)\,,
\end{align}
where for $n\ge1$ and $i,j\ge0$, $i\neq j$, we set 
\[
\lambda_{ni}=1-\frac{\sigma_Z^2}{\sigma_n^2}(A_i-A_{i-n})^2 , \ \ 
\lambda_{nij}=1-\frac{\sigma_Z^2}{\sigma_n^2}\left[(A_i-A_{i-n})^2+(A_j-A_{j-n})^2\right]\,.
\]
\end{lemma}
\begin{proof}
It is an elementary conclusion from \eqref{e:moment.eq} that, for each
$1\le i\le\kappa$, 
 \begin{equation}
\label{lpm.newl0.eq2} \left(\int_\bbr e^{\delta z}
  \,F_z(dz)\right)^{-1}\int_\bbr  z^ie^{\delta z} \,F_z(dz)=
\sigma_Z^iE\left[(G+\delta
  \sigma_Z)^i\right]+O\left(|\delta|^{\kappa-i+1}\right) 
\end{equation}
as
$\delta\to 0$, where $G$ is a standard Gaussian random variable.

Let $(\hat Z_{ni}:\, n\ge1,i\ge0)$ be a family of independent random
variables with each $\hat Z_{ni}\sim G_{(A_i-A_{i-n})\zeta_n}$, 
so that for $n\ge1$ and $i,j\ge0$, $i\neq j$ we have 
\[
\hat S_n\eid\sum_{k=0}^\infty\left(A_k-A_{k-n}\right)\hat Z_{nk}\,,
\]
$$
\hat
S_n(i)\eid\sum_{k\in\{0,1,2,\ldots\}\setminus\{i\}}\left(A_k-A_{k-n}\right)\hat
Z_{nk}\,, 
$$
$$
\hat
S_n(i,j)\eid\sum_{k\in\{0,1,2,\ldots\}\setminus\{i,j\}}\left(A_k-A_{k-n}\right)\hat
Z_{nk}\,.
$$
 
Let now  $(G_{ni}:\, n\ge1,i\ge0)$ be a collection of independent random
variables, also  independent of $(\hat Z_{ni}:n\ge1,i\ge0)$,  where
\[
G_{ni}\sim N\left((A_i-A_{i-n})\zeta_n\sigma_Z^2\,,\sigma_Z^2\right)\,, \text{ for all }n\ge1,\,i\ge0\,.
\]
It follows from Lemma \ref{lpm.l0} and \eqref{e:A.incr} that  
\eqref{lpm.newl0.eq2} can be reformulated as 
\begin{equation}
\label{lpm.newl0.eq3} E\left(\hat 
  Z_{nj}^i\right)-E\left(G_{nj}^i\right)=O\left(|A_j-A_{j-n}|^{\kappa-i+1}
  n^{-2(1-\alpha)(\kappa-i+1)}\right) 
\end{equation}
uniformly in $j\ge0$ and $1\le i\le\kappa$. For a fixed $t\in\bbr$ we
use telescoping to write 
\begin{align}
\label{lpm.newl0.eq4}&\left| E\exp\left\{i\left(t\sigma_n^{-1}\sum_{j=0}^\infty
(A_j-A_{j-n})G_{nj}\right)\right\}
                       - E\exp\left\{i\left(t\sigma_n^{-1}\hat
                       S_n\right)\right\}\right| \\
\nonumber&\le\sum_{j=0}^\infty\left|E\exp\left\{i\left(t\sigma_n^{-1}\left(\sum_{k=0}^{j-1}(A_j-A_{j-n})\hat Z_{nj}+\sum_{k=j}^\infty(A_j-A_{j-n})G_{nj}\right)\right)\right\}\right.\\
\nonumber &\left.\,\,\,\,-E\exp\left\{i\left(t\sigma_n^{-1}\left(\sum_{k=0}^{j}(A_j-A_{j-n})\hat Z_{nj}+\sum_{k=j+1}^\infty(A_j-A_{j-n})G_{nj}\right)\right)\right\}\right|\,.
\end{align}

Fix $j\ge0$ and denote 
\begin{align*}
U&=t\sigma_n^{-1}\left(\sum_{k=0}^{j-1}(A_j-A_{j-n})\hat 
  Z_{nj}+\sum_{k=j+1}^\infty(A_j-A_{j-n})G_{nj}\right), \\
V&=t\sigma_n^{-1}(A_j-A_{j-n})G_{nj}\,,
\end{align*}
so that by expanding in the Taylor series around $U$, 
\begin{align*}
&E\exp\left\{i\left(t\sigma_n^{-1}\left(\sum_{k=0}^{j-1}(A_j-A_{j-n})\hat
                 Z_{nj}+\sum_{k=j}^\infty(A_j-A_{j-n})G_{nj}\right)\right)\right\}
\\
                 &=Ee^{i(U+V)}
=\sum_{m=0}^\kappa\frac{i^m}{m!}E\left(V^m\right) Ee^{iU}
                   + R_1\,, 
\end{align*}
with $|R_1|\le E(|V|^{\kappa+1})/(\kappa+1)!$. Similarly, 
\begin{align*}
&E\exp\left\{i\left(t\sigma_n^{-1}\left(\sum_{k=0}^{j}(A_j-A_{j-n})\hat Z_{nj}+\sum_{k=j+1}^\infty(A_j-A_{j-n})G_{nj}\right)\right)\right\}\\
&=\sum_{m=0}^\kappa\frac{i^m}{m!}E\left(W^m\right)Ee^{iU}+R_2\,,
\end{align*}
with $|R_2|\le E(|W|^{\kappa+1})/(\kappa+1)!$, where 
\[
W=(A_j-A_{j-n})\hat Z_{nj}\,. 
\]
We conclude that 
\begin{align}
\nonumber&\left|E\exp\left\{i\left(t\sigma_n^{-1}\left(\sum_{k=0}^{j-1}(A_j-A_{j-n})
\hat Z_{nj}+\sum_{k=j}^\infty(A_j-A_{j-n})G_{nj}\right)\right)\right\}\right.\\
\nonumber&\left.\,\,\,\,-E\exp\left\{i\left(t\sigma_n^{-1}\left(\sum_{k=0}^{j}(A_j-A_{j-n})
\hat Z_{nj}+\sum_{k=j+1}^\infty(A_j-A_{j-n})G_{nj}\right)\right)\right\}\right| \\
\nonumber&\le\sum_{i=1}^\kappa\frac{|t|^i}{i!}\left|(A_j-A_{j-n})^i\sigma_n^{-i}E\left(
\hat Z_{nj}^i-G_{nj}^i\right)\right| \\
\label{lpm.newl0.eq5}&\,\,\,\,\,\,\,  +\frac{|t|^{\kappa+1}}{(\kappa+1)!}\left|A_j-A_{j-n}\right|^{\kappa+1}\sigma_n^{-(\kappa+1)}E\left(|G_{nj}|^{\kappa+1}+|\hat Z_{nj}|^{\kappa+1}\right)\,.
\end{align}
Note that by \eqref{lpm.l1.eqcube} and Lemma \ref{lpm.l0}, 
\begin{align*}
&\sigma_n^{-(\kappa+1)}\sum_{j=0}^\infty\left|A_j-A_{j-n}\right|^{\kappa+1}E\left(|G_{nj}|^{\kappa+1}+|\tilde Z_{nj}|^{\kappa+1}\right)\\
&=O\left(n^{-(\kappa-1)/2}\right)
=o\left(n^{1/2-\kappa(1-\alpha)}\right)\,.
\end{align*}
For $1\le i\le\kappa$ we  use, in addition. \eqref{lpm.newl0.eq3} to write
\begin{align*}
&\sigma_n^{-i}\sum_{j=0}^\infty\left|(A_j-A_{j-n})^iE\left(\tilde Z_{nj}^i-G_{nj}^i\right)\right|\\
&=O\left(n^{-\kappa(1-\alpha)+\alpha-i(\alpha-1/2)}\right) =O\left(n^{1/2-\kappa(1-\alpha)}\right)\,.
\end{align*}
Putting these bounds  into \eqref{lpm.newl0.eq5} we obtain 
\begin{equation*}
E\left(e^{\iota t\sigma_n^{-1}\tilde
    S_n}\right)=\phi_G\left(\sigma_n^{-1}n\vep;1;t\right)+O\left(n^{1/2-\kappa(1-\alpha)}\left(1+|t|^{\kappa+1}\right)\right) 
\end{equation*}
uniformly for $t\in\bbr$, which is equivalent to
\eqref{lpm.newl0.claim1}.  The argument for \eqref{lpm.newl0.claim2}
and \eqref{lpm.newl0.claim3} is the same. 
\end{proof}

By the assumption \eqref{eq.chf}, for large $n$, 
 the random variables $\sigma_n^{-1}(\hat S_n-n\vep)$, $\sigma_n^{-1}(\hat 
S_n(i)-n\vep)$ and $\sigma_n^{-1}(\hat S_n(i,j)-n\vep)$ have densities
which we denote by $f_n$, $f_{ni}$ 
and $f_{nij}$, correspondingly. 

\begin{lemma}\label{lpm.newl1}
  Suppose that \eqref{e:moment.eq}  and\eqref{eq.chf} hold. Then for large $n$, the densities
 $f_{ni}$ and $f_{nij}$ are twice differentiable. Furthermore, as
 $n\to\infty$, 
\begin{align}
\label{lpm.newl1.eq0}f_{ni}(0)&=(2\pi)^{-1/2}+o\left(n^{1-2\alpha}\right)\,,\\
  \label{lpm.newl1.eq-1}f_{ni}'(0)&=o\left(n^{1/2-\alpha}\right)
\end{align}
uniformly in $i$, and for some $n_0\in\bbn$,
\begin{equation}\label{lpm.newl1.eq-2}
\sup\left\{\left|f_{ni}''(x)\right|:n\ge n_0,\,i\ge0,\,x\in\bbr\right\} < \infty\,.
\end{equation}
All three statements also hold if $f_{ni}$ is replaced by
$f_{nij}$, $i<j$.
Finally,  as $n\to\infty$,
\begin{equation}
\label{lpm.newl1.eq4} \sup_{x\in\bbr}\left|f_n(x)-(2\pi)^{-1/2}e^{-x^2/2}\right|=o\left(n^{1-2\alpha}\right)\,.
\end{equation}
\end{lemma}
\begin{proof}
We start with the proof of \eqref{lpm.newl1.eq4} which would follow from the inversion formula for densities once it is shown that
\[
\int_{-\infty}^\infty \left|\phi_n(t)-\phi_G(0;1;t)\right|dt=o\left(n^{1-2\alpha}\right)\,.
\]
By Lemma \ref{lpm.newl0} and \eqref{eq.kappa}, 
\begin{align*}
\int_{-\log n}^{\log n}\left|\phi_n(t)-\phi_G(0;1;t)\right|dt&=O\left(n^{1/2-\kappa(1-\alpha)}(\log n)^{\kappa+2}\right)
=o\left(n^{1-2\alpha}\right)\,. 
\end{align*}
Furthermore, 
\[
\int_{[-\log n,\log n]^c}\phi_G(0;1;t)\, dt =O\left(e^{-(\log n)^2/2}\right)=o\left(n^{1-2\alpha}\right)\,,
\]
Thus, \eqref{lpm.newl1.eq4} will follow once we show that 
\begin{equation}
\label{lpm.newl1.eq5} \int_{[-\log n,\log n]^c}
\left|\phi_n(t)\right|dt =o\left(n^{1-2\alpha}\right)\,.
\end{equation}

With $(\hat Z_{ni}:n\ge1,i\ge0)$  as above, we set 
\[
U_{ni}=\sigma_n^{-1}(A_i-A_{i-n})\left[\hat Z_{ni}-E(\hat
  Z_{ni})\right], \, n\ge1,\, i\ge0\,, 
\]
so that 
\begin{equation}
\label{lpm.newl1.eq9}\left|\phi_n(t)\right|=\prod_{i=0}^\infty\left|E\left(e^{\iota tU_{ni}}\right)\right|\,,n\ge1,t\in\bbr\,.
\end{equation}
Set
\[
H(x,t)=\left(\int_{-\infty}^\infty e^{xz}f_Z(z)\, dz\right)^{-1}
\int_{-\infty}^\infty e^{(x+\iota t)z}f_Z(z)\, dz ,\, (x,t)\in\bbr^2\,,
\]
which is a characteristic function for any fixed $x$.  A consequence of that is
$\partial |H(x,t)|/\partial t|_{t=0}\leq 0$ for any
$x\in\bbr$. Furthermore, 
\[
\frac{\partial^2}{\partial t^2} | H(0,t)|\Bigr|_{t=0}=-\sigma_Z^2<0
\]
and by continuity of the second partial derivative we conclude that
there is $\delta_0>0$ such that
\[
\frac{\partial^2}{\partial t^2} | H(x,t)|\Bigr|<0 \ \text{whenever} \
0\leq |t|,|x|\leq \delta_0.
\]
That means we also have 
\begin{equation}
\label{lpm.newl1.eq7} \frac{\partial}{\partial t} | H(x,t)|\Bigr|\leq 0 \ \text{whenever} \
0\leq |t|,|x|\leq \delta_0.
\end{equation}
We may and will choose $\delta_0\in (0,\theta_0]$, with $\theta_0$ as
in \eqref{eq.chf}. By \eqref{eq.chf} we can appeal to \eqref{f.rl} to conclude that 
\[
\lim_{t\to\infty}\sup_{|x|\le\delta_0}|H(x,t)|=0\,.
\]
Thus, there is $M>0$ large enough so that
$$
\sup_{t>M,|x|\le\delta_0}|H(x,t)|<1.
$$
Since by continuity of $H$ and compactness we have
$$
\sup_{\delta_0\le t\le M,|x|\le\delta_0}|H(x,t)|<1, 
$$
it follows that 
\[
\eta=\sup_{t\ge\delta_0,|x|\le\delta_0}|H(x,t)|<1\,.
\]
The continuity argument also shows that there is
$\delta_1\in(0,\delta_0]$ such that 
\[
\min_{|x|\le\delta_0}|H(x,\delta_1)|\ge\eta\,.
\]
Therefore, for $|x|\le\delta_0$ and $0\le t\le\delta_1$,  \eqref{lpm.newl1.eq7} implies that
\begin{align*}
|H(x,t)|& \ge|H(x,\delta_1)|\ge\eta \ge\sup_{s\ge\delta_0}|H(x,s)|\,.
\end{align*}
Since by \eqref{lpm.newl1.eq7} we also have 
\[
|H(x,t)|=\sup_{s\in[t,\delta_0]}|H(x,s)|\,,
\]
we conclude that 
\begin{equation}
\label{lpm.newl1.eq8} |H(x,t)|=\sup_{s\ge t}|H(x,s)|, \,
|x|\le\delta_0, \, 0\le t\le\delta_1\,.
\end{equation}

By \eqref{lpm.newl1.eq9} 
\begin{align}\notag 
 |\phi_n(t)|&\le\left|E(e^{\iota
                       tU_{nn}})\right|\prod_{i=[n/2]}^{n-1}\left|E(e^{\iota
                       tU_{ni}})\right| \\
 \label{lpm.newl1.eq10} &=\left|E(e^{\iota tU_{nn}})\right|\prod_{i=[n/2]}^{n-1}\left|H\left(\zeta_nA_i,\sigma_n^{-1}A_it\right)\right|\,.
\end{align}

It follows from Lemma \ref{lpm.l0} that 
there exists $s_0>0$ such that for all $n$ large enough, 
\[
A_i\ge s_0\sigma_n n^{-1/2}, \, [n/2]\le i\le n-1\,.
\]
Thus, for $n$ large enough and $t\ge\log n$,  \eqref{lpm.newl1.eq8} implies that 
\begin{align*}
\prod_{i=[n/2]}^{n-1}\left|H\left(\zeta_nA_i,\sigma_n^{-1}A_it\right)\right| \le \prod_{i=[n/2]}^{n-1}\left|H\left(\zeta_nA_i,s_0n^{-1/2}\log n\right)\right| \,.
\end{align*}
Since any partial derivative of $H$ is bounded on a compact set,
we can use the bound \eqref{f.chf} to conclude that there exists
$s_1>0$ such that 
\[
\sup_{|x|\le\delta_0}|H(x,t)|\le (1-s_1t^2)^{1/2} ,\  0\le t\le 1\,.
\]
Thus, there is $s_2>0$ such that for all large $n$ and all $t\geq
\log n$ we have 
\[
\prod_{i=[n/2]}^{n-1}\left|H\left(\zeta_nA_i,\sigma_n^{-1}A_it\right)\right|\le\left(1-s_0^2s_1n^{-1}(\log n)^2\right)^{n/4}=O\left(e^{-s_2(\log n)^2}\right)\,.
\]
Using this bound in \eqref{lpm.newl1.eq10}, and appealing to
\eqref{eq.chf} we obtain 
\begin{align*}
\int_{\log n}^\infty \left|\phi_n(t)\right|dt  &=O\left(e^{-s_2(\log
     n)^2} \right)\int_{\log n}^\infty \left|E\left(e^{i tU_{nn}}\right)\right|dt\\
&=O\left(n^{1/2}e^{-s_2(\log n)^2}\right)
=o\left(n^{1-2\alpha}\right)\,.
\end{align*}
Since we can switch from $t$ to $-t$, \eqref{lpm.newl1.eq5} follows,
which establishes \eqref{lpm.newl1.eq4}.  

A similar calculation with the aid of \eqref{lpm.newl0.claim2} shows that
\[
f_{ni}(0)=\left(2\pi\lambda_{ni}\right)^{-1/2}\exp\left(-\sigma_n^{-2}n^2\vep^2(\lambda_{ni}-1)^2/2\lambda_{ni}\right)+o\left(n^{1-2\alpha}\right)\,,
\]
uniformly in $i\ge0$.  Since $\lambda_{ni}-1=O(1/n)$
uniformly in $i\ge0$, it follows that
\begin{align*}
\lambda_{ni}^{-1/2}\exp\left(-\sigma_n^{-2}n^2\vep^2(\lambda_{ni}-1)^2/2\lambda_{ni}\right)&=1+O\left(n^{-1}+\sigma_n^{-2}\right)\\
&=1+o\left(n^{1-2\alpha}\right)\,,
\end{align*}
uniformly for $i\ge0$,  which proves \eqref{lpm.newl1.eq0}.   
For \eqref{lpm.newl1.eq-2} we write 
\[
f_{nk}''(x)=-(2\pi)^{-1/2}\int_{-\infty}^\infty e^{-i
  tx}t^2\phi_{nk}(t)\, dt 
\]
and repeat the arguments used above in the proof of
\eqref{lpm.newl1.eq4}, applying \eqref{lpm.newl0.claim2} and
the full force of the assumption \eqref{eq.chf}.  

Finally, for \eqref{lpm.newl1.eq-1} we use the identity
$$
f_{nk}'(0)=-i(2\pi)^{-1/2}\int_{-\infty}^\infty  t\phi_{nk}(t)\, dt. 
$$
Since 
\begin{align*}
\left|\int_{-\infty}^\infty
  t\,\phi_G\left(\sigma_n^{-1}n\vep(\lambda_{nk}-1);\lambda_{nk};t\right)dt\right|
  =O\left(\sigma_n^{-1}\right)
=o\left(n^{1/2-\alpha}\right)\,,
\end{align*}
uniformly in $k\ge0$,  \eqref{lpm.newl1.eq-1} follows.

The arguments with $f_{nij}$ replacing $f_{ni}$ are similar.   This
completes the proof. 
\end{proof}

The next lemma tackles certain expectations 
conditionally on $E_0$; its statement should be compared to
\eqref{e:P.E0}. 
\begin{lemma}\label{lpm.ct}
  Suppose that \eqref{e:moment.eq}  and\eqref{eq.chf} hold. Then
\begin{align} \label{lpm.newl2.eq5}
E\left(Z_{n-i-1}\one(E_0)\right)
 =K_n\left[ \int_{-\infty}^\infty z\, G_{\zeta_n(A_i-A_{i-n})}(dz) + o\left(\zeta_n^{-1}\sigma_n^{-2}|A_i-A_{i-n}|\right)\right]
\end{align}
and 
\begin{align} \label{e:cros..pr}
&E\left(Z_{n-i-1}Z_{n-j-1}\one(E_0)\right)\\
\notag &=K_n\Biggl(\int_{-\infty}^\infty z_1\, G_{\zeta_n(A_i-A_{i-n})}(dz_1)
\int_{-\infty}^\infty z_2\, G_{\zeta_n(A_i-A_{i-n})}(dz_2)\\
\notag &+o\left(\sigma_n^{-2}|(A_i-A_{i-n})(A_j-A_{j-n})|\right)\Biggr), \
   n\to\infty, 
\end{align}
uniformly for $i,j\ge0$ with $i\neq j$, where 
\begin{equation}
\label{lpm.ct.eq2}K_n=(2\pi)^{-1/2}\zeta_n^{-1}\sigma_n^{-1}e^{-n\vep\zeta_n}E\left(e^{\zeta_nS_n}\right),
\, n\ge1\,.
\end{equation}
\end{lemma}
\begin{proof}
We only prove \eqref{e:cros..pr}; the proof of \eqref{lpm.newl2.eq5}
is similar and easier. Write 
\begin{align*}
&E\left(Z_{n-i-1}Z_{n-j-1}\one(E_0)\right)\\
&=\int_{-\infty}^\infty z_1\, F_Z(dz_1)\int_{-\infty}^\infty z_2\, F_Z(dz_2) \\
&\hspace{1in} P\left(S_n'(i,j)\ge n\vep-(A_i-A_{i-n})z_1-(A_j-A_{j-n})z_2\right) \\
&=\sigma_n^{-1}E\left(e^{\zeta_nS_n'(i,j)}\right)
    \int_{-\infty}^\infty z_1\, F_Z(dz_1)
\int_{-\infty}^\infty z_2\, F_Z(dz_2)\\
&\hspace{1in} \int_{n\vep-(A_i-A_{i-n})z_1-(A_j-A_{j-n})z_2}^\infty f_{nij}
                                          \bigl(s-n\vep)/\sigma_n\bigr)
                                          e^{-\zeta_ns} \, ds.
\end{align*}
We adopt the convention $\int_a^b\equiv-\int_b^a$, and denote
\begin{align*}
c_{nij}&=\zeta_n^{-1}\sigma_n^{-1}e^{-n\vep\zeta_n}E\left(e^{\zeta_nS_n'(i,j)}\right)\\
&=K_n(2\pi)^{1/2}\left(\int_{-\infty}^\infty e^{\zeta_n(A_i-A_{i-n})z}F_Z(dz)
\int_{-\infty}^\infty e^{\zeta_n(A_j-A_{j-n})z}F_Z(dz)\right)^{-1}.
\end{align*}
Changing the variable and using the fact that $EZ=0$, we obtain 
\begin{align}
\nonumber&E\left(Z_{n-i-1}Z_{n-j-1}\one(E_0)\right)\\
\nonumber&=c_{nij} \int_{-\infty}^\infty z_1\, F_Z(dz_1)
\int_{-\infty}^\infty z_2\, F_Z(dz_2)\\
\nonumber&\hspace{.5in}
           \int_{0}^{\zeta_n(A_i-A_{i-n})z_1+\zeta_n(A_j-A_{j-n})z_2}e^x
           f_{nij}\bigl(-x/(\sigma_n\zeta_n)\bigr)\, dx\\
 \label{lpm.ct.eq4}&= c_{nij} \int_{-\infty}^\infty z_1\, F_Z(dz_1)
\int_{-\infty}^\infty z_2\, F_Z(dz_2)\\
\nonumber &\hspace{.5in}\Biggl[ \int_{0}^{\zeta_n(A_i-A_{i-n})z_1+\zeta_n(A_j-A_{j-n})z_2}e^x
           f_{nij}\bigl(-x/(\sigma_n\zeta_n)\bigr)\, dx\\
\nonumber &\hspace{.5in} - \int_0^{\zeta_n(A_i-A_{i-n})z_1} e^x
           f_{nij}\bigl(-x/(\sigma_n\zeta_n)\bigr)\, dx\\
\nonumber &\hspace{.5in}-  \int_0^{\zeta_n(A_j-A_{j-n})z_2}e^x
           f_{nij}\bigl(-x/(\sigma_n\zeta_n)\bigr)\, dx \Biggr]\,.
\end{align}

For fixed  $z_1,z_2\in\bbr$,  the expression inside the square
brackets can be rewritten as 
\begin{align*}
&\left(e^{\zeta_n(A_i-A_{i-n})z_1}-1\right)\int_0^{\zeta_n(A_j-A_{j-n})z_2}e^x\\
&\hskip1.5in
            f_{nij}\bigl(-(x+\zeta_n(A_i-A_{i-n})z_1)/(\sigma_n\zeta_n)\bigr)\,
  dx\\
& + \int_0^{\zeta_n(A_j-A_{j-n})z_2}e^x\\
&\hskip0.5in\Biggl[ 
f_{nij}\bigl(-(x+\zeta_n(A_i-A_{i-n})z_1)/(\sigma_n\zeta_n)\bigr)- f_{nij}\bigl(-x/(\sigma_n\zeta_n)\bigr)\Biggr]dx.  
\end{align*}
By Taylor's theorem, 
\[
f_{nij}\left(-\,\frac{x+\zeta_n(A_i-A_{i-n})z_1}{\sigma_n\zeta_n}\right)=f_{nij}(0)-\,\frac{x+\zeta_n(A_i-A_{i-n})z_1}{\sigma_n\zeta_n}f_{nij}'(0)
\]
\[
+O\left(\frac{(x+\zeta_n(A_i-A_{i-n})z_1)^2}{\sigma_n^2\zeta_n^2}\|f_{nij}''\|_\infty\right)\,.
\]
Using this and \eqref{e:A.incr}, straightforward algebra gives us
\begin{align*}
&  \int_0^{\zeta_n(A_j-A_{j-n})z_2}e^x
    f_{nij}\bigl(-(x+\zeta_n(A_i-A_{i-n})z_1)/(\sigma_n\zeta_n)\bigr)\,  dx
\\
& =f_{nij}(0)\left(e^{\zeta_n(A_j-A_{j-n})z_2} - 1\right)\\
&+O\Biggl(e^{\zeta_n|A_j-A_{j-n}||z_2|}\Bigl(|f_{nij}'(0)|\sigma_n^{-1}\zeta_nn^{1-\alpha}|A_j-A_{j-n}||z_2|\bigl(|z_1|+|z_2|\bigr)\\
&\hspace{0.5in}+\|f_{nij}''\|_\infty\sigma_n^{-2}\zeta_nn^{2-2\alpha}|A_j-A_{j-n}||z_2|\bigl(|z_1|+|z_2|\bigr)^2\Bigr)\Biggr)\,.
\end{align*}
The obvious inequality $|e^x-1|\leq
|x|e^{|x|}$ for $x\in\bbr$ along with Lemma
\ref{lpm.newl1}  now show that 
\begin{align*}
\nonumber &\left(e^{\zeta_n(A_i-A_{i-n})z_1}-1\right)
\int_0^{\zeta_n(A_j-A_{j-n})z_2}e^x\\
&\hskip1.5in    f_{nij}\bigl(-(x+\zeta_n(A_i-A_{i-n})z_1)/(\sigma_n\zeta_n)\bigr)\,  dx\\
&= f_{nij}(0)\left(e^{\zeta_n(A_i-A_{i-n})z_1}-1\right)\left(e^{\zeta_n(A_j-A_{j-n})z_2}-1\right)\\
\nonumber& +o\Bigl(\sigma_n^{-2}\left|(A_i-A_{i-n})(A_j-A_{j-n})z_1z_2\right|\left(|z_1|+|z_2|\right)^2\\
&\hskip1.5in e^{\zeta_n(|A_i-A_{i-n}||z_1|+|A_j-A_{j-n}||z_2|)}\Bigr)\,,
\end{align*}
uniformly for $i,j\ge0$ with $i\neq j$ and $z_1,z_2\in\bbr$.  

Treating in a similar manner the second term, we conclude that the
expression inside the square brackets in the right hand side of
\eqref{lpm.ct.eq4} equals
\begin{align*}
&f_{nij}(0) \left(e^{\zeta_n(A_i-A_{i-n})z_1}-1\right) \left(e^{\zeta_n(A_j-A{j-n})z_2}-1\right)\\
&+o\Bigl(\sigma_n^{-2}\left| (A_i-A_{i-n})(A_j-A_{j-n})\right|(1+|z_1|^3)(1+|z_2|^3)\\
&\hskip1.5in 
e^{\zeta_n|(A_i-A_{i-n})z_1|+\zeta_n|(A_j-A_{j-n})z_2|}\Bigr)\,,
\end{align*}
uniformly for $i,j\ge0$ with $i\neq j$ and $z_1,z_2\in\bbr$, and
substitution into \eqref{lpm.ct.eq4} gives us 
\begin{align}
\nonumber&E\left(Z_{n-i-1}Z_{n-j-1}\one(E_0)\right)\\
\nonumber&=c_{nij}\Biggl[f_{nij}(0)\int_{-\infty}^\infty
  z_1e^{\zeta_n(A_i-A_{i-n})z_1}F_Z(dz_1)
  \int_{-\infty}^\infty  z_2e^{\zeta_n(A_j-A_{j-n})z_2}F_Z(dz_2)\\
\nonumber&\hspace{.5in} +o\Bigl(\sigma_n^{-2}\left| (A_i-A_{i-n})(A_j-A_{j-n})\right| \Bigr)\Biggr]\end{align}
\begin{align}
\label{lpm.ct.eq1}&=K_n(2\pi)^{1/2}f_{nij}(0)\int_{-\infty}^\infty
                    z_1\, G_{\zeta_n(A_i-A_{i-n})}(dz_1)
                    \int_{-\infty}^\infty z_2\, G_{\zeta_n(A_j-A_{j-n})}(dz_2)\\
\nonumber&\hspace{.5in}+c_{nij}o\Bigl(\sigma_n^{-2}\left| (A_i-A_{i-n})(A_j-A_{j-n})\right| \Bigr)\,,
\end{align}
as $n\to\infty$,  uniformly for $i,j\ge0$ with $i\neq j$.  Recalling
that $EZ=0$, we see that 
\[
\int_{-\infty}^\infty z_1\,G_{\zeta_n(A_i-A_{i-n})}(dz_1)=O\left(\zeta_n(A_i-A_{i-n})\right)\,,
\]
and likewise for the second integral in \eqref{lpm.ct.eq1}. Since
$K_n=O(c_{nij})$, the claim \eqref{e:cros..pr}  follows from Lemma
\ref{lpm.newl1}.  
\end{proof}

The next lemma is an important step in the proof of the main result;
the previous lemmas \ref{lpm.newl0}, \ref{lpm.newl1} and \ref{lpm.ct}
are needed for this lemma. We denote 
\begin{equation}
\label{eq.defyni}Y_{ni}=Z_{n-i-1} -
\left(1+\zeta_n^{-2}\sigma_n^{-2}\right) \int_{-\infty}^\infty z\,
G_{\zeta_n(A_i-A_{i-n})}(dz), \, i\in\bbz, \, n\ge1\,.
\end{equation}

\begin{lemma}\label{lpm.newl2}
 Suppose that \eqref{e:moment.eq}  and\eqref{eq.chf} hold. Then 
\begin{equation}\label{lpm.newl2.eq1}
 \sup_{n\ge1,i\ge0}E\left(Y_{ni}^2\bigr|E_0\right)<\infty\,,
\end{equation}
and
\begin{align}\label{lpm.newl2.eq2}
E\left(Y_{ni}Y_{nj}\bigr|E_0\right) =-\sigma_n^{-2}\sigma_Z^4\left(A_i-A_{i-n}\right)\left(A_j-A_{j-n}\right)\left(1+o(1)\right)
\end{align}
as $n\to\infty$,  uniformly in $i,j\ge0$ with $i\neq j$.
\end{lemma}
\begin{proof}
We  prove \eqref{lpm.newl2.eq2}; the proof of \eqref{lpm.newl2.eq1} is
similar (and much easier). We write
\begin{align*}
P(E_0)&=K_n(2\pi)^{1/2}\int_0^\infty e^{-x}f_n\bigl( 
        x/(\zeta_n\sigma_n)\bigr)\, dx, 
\end{align*}
with $K_n$ as in \eqref{lpm.ct.eq2}.  By \eqref{lpm.newl1.eq4} and
simple integration, 
\begin{align}\label{lpm.newl2.eq0} 
P(E_0)=&K_n
(2\pi)^{1/2}\left[o\left(\zeta_n^{-2}\sigma_n^{-2}\right)+(2\pi)^{-1/2}\int_0^\infty
  \exp\bigl(-x-x^2/(2\zeta_n^2\sigma_n^2)\bigr)\, dx\right] \\
\notag =&K_n\left[1-\zeta_n^{-2}\sigma_n^{-2}\left(1+o(1)\right)\right], \, n\to\infty\,.
\end{align}
In combination with \eqref{e:cros..pr}  this means that 
\begin{align*}
&E\left(Z_{n-i-1}Z_{n-j-1}\one(E_0)\right)P(E_0)\\
&=K_n^2\Biggl(\left(1-\zeta_n^{-2}\sigma_n^{-2}\right)\int_{-\infty}^\infty
           z_1\, G_{\zeta_n(A_i-A_{i-n})}(dz_1) \int_{-\infty}^\infty
           z_2\, G_{\zeta_n(A_j-A_{j-n})}(dz_2)\\
&\hspace{0.5in}
           +o\left(\sigma_n^{-2}|(A_i-A_{i-n})(A_j-A_{j-n})|\right)\Biggr),
           \, n\to\infty,
\end{align*}
uniformly in $i,j\ge0$ with $i\neq j$.  Since by 
\eqref{lpm.newl2.eq5}, 
\begin{align*}
&E\left(Z_{n-i-1}\one(E_0)\right)E\left(Z_{n-j-1}\one(E_0)\right)\\
&=K_n^2\int_{-\infty}^\infty z_1\,G_{\zeta_n(A_i-A_{i-n})}(dz_1)\int_{-\infty}^\infty z_2\,G_{\zeta_n(A_j-A_{j-n})}(dz_2)\\
&\hspace{.5in}+o\left(K_n^2\sigma_n^{-2}|A_i-A_{i-n}||A_j-A_{j-n}|\right)\,,
\end{align*}
we conclude that
\begin{align*}
&E\left(Z_{n-i-1}Z_{n-j-1}\one(E_0)\right)P(E_0)-E\left(Z_{n-i-1}\one(E_0)\right)E\left(Z_{n-j-1}\one(E_0)\right)\\
&=-K_n^2\zeta_n^{-2}\sigma_n^{-2}\int_{-\infty}^\infty z_1\,G_{\zeta_n(A_i-A_{i-n})}(dz_1)
\int_{-\infty}^\infty z_2\, G_{\zeta_n(A_j-A_{j-n})}(dz_2)\\
&\hspace{.5in}+o\left(K_n^2\sigma_n^{-2}|A_i-A_{i-n}||A_j-A_{j-n}|\right)\\
&=-K_n^2\sigma_n^{-2}\sigma_Z^4(A_i-A_{i-n})(A_j-A_{j-n})\left(1+o(1)\right)
\end{align*}
as $n\to\infty$, uniformly in $i,j\ge0$ with $i\neq j$.  Dividing both
sides by $P(E_0)^2$ and using \eqref{lpm.newl2.eq0},  we obtain 
\begin{align}
\label{lpm.newl2.eq3}&E\left[\bigl(Z_{n-i-1}-E(Z_{n-i-1}|E_0)\bigr)\bigl(Z_{n-j-1}-E(Z_{n-j-1}|E_0)\bigr)\Bigr|E_0\right]\\
\nonumber&=-\sigma_n^{-2}\sigma_Z^4(A_i-A_{i-n})(A_j-A_{j-n})\left(1+o(1)\right)\,,
\end{align}
as $n\to\infty$, again uniformly for $i,j\ge0$ with $i\neq j$.  Since
by  \eqref{lpm.newl2.eq0} with \eqref{lpm.newl2.eq5} 
\begin{align*}
E\left(Z_{n-i-1}|E_0\right)
&=\left(1+\zeta_n^{-2}\sigma_n^{-2}\right)\int_{-\infty}^\infty z\, G_{\zeta_n(A_i-A_{i-n})}(dz)\\
&\hskip2in+o\left(\zeta_n^{-1}\sigma_n^{-2}|A_i-A_{i-n}|\right)\,,
\end{align*}
with a similar statement for $Z_{n-j-1}$, 
\eqref{lpm.newl2.eq3} implies \eqref{lpm.newl2.eq2}. 
 \end{proof}

We proceed with establishing conditional distributional limits of certain
truncated sums. 
\begin{lemma}\label{lpm.l4}
 Suppose that \eqref{e:moment.eq}  and\eqref{eq.chf} hold. For
 $0<\delta<L$ denote  
\begin{align}
\label{lpm.l4.eq2}S_n(j,\delta,L)=&\sum_{i=[n^\beta\delta]}^{[n^\beta
  L]-1}(A_{i+j}-A_{i})Y_{ni}+\sum_{i=n-j}^{n-1}(A_{i+j}-A_{i+j-n}-A_i)Y_{ni}\\
\nonumber+&\sum_{i=n}^{n+[n^\beta L]}(A_{i+j}-A_{i+j-n}-A_i+A_{i-n})Y_{ni}, \ n\geq 1, \, j\geq 0.
\end{align}
With the overshoot $T_n^*$ as in \eqref{lpm.eq.defTn}, we have, 
conditionally on $E_0$, 
\begin{align}
\nonumber&\left(\zeta_nT_n^*,\left(n^{2\alpha-2}S_n([n^\beta
           t],\delta,L), \, t\ge0\right)\right)\\
\label{lpm.l4.eq3}&\Rightarrow \left(T_0,  \left( (1-\alpha)^{-1}\sigma_Z \left( \int_\delta^L\left[(s+t)^{1-\alpha}-s^{1-\alpha}\right] dB_1(s)\right. \right.\right.\\
\nonumber&\,\,\,\,\,\, \,\,\,\,  +\left.\left.\left.
           \int_0^t(t-s)^{1-\alpha} dB_2(s)+\int_0^L
           \left[s^{1-\alpha}-(s+t)^{1-\alpha}\right]dB_3(s)\right),
           \,            t\ge0\right)\right) 
\end{align}
in finite dimensional distributions as $n\to\infty$, where $T_0$ is a standard
exponential random variable independent of independent standard
Brownian motions $B_1,B_2,B_3$,
 \end{lemma}

\begin{proof}
For $n\ge1$ and $t\ge0$ we write 
\begin{align*}
\xi_n^{1\circ}(t)&=\sum_{i=1}^{[n^\beta t]}A_iY_{ni}, \ \ 
\xi_n^{2\circ}(t)=\sum_{i=n-[n^\beta t]}^{n-1}A_iY_{ni}, 
\end{align*}
\begin{align*}
\xi_n^{3\circ}(t)=\sum_{i=n+1}^{n+[n^\beta
                   t]}\left(A_i-A_{i-n}\right)Y_{ni}. 
\end{align*}
It follows from Lemma \ref{lpm.l3} that, conditionally on $E_0$,
\begin{align}\label{lpm.l4.eq1}
 &\Biggl(\zeta_nT_n^*,\left(n^{2\alpha-2}\xi_n^{1\circ}(t):t\ge0\right),\left(n^{\alpha-\beta/2-1}\xi_n^{2\circ}(t):t\ge0\right),\\
\notag                &\hskip2in 
  \left(n^{\alpha-\beta/2-1}\xi_n^{3\circ}(t):t\ge0\right)\Biggr)
\end{align}
\begin{align*}
  \Rightarrow &\,\Bigl(T_0,\bigl(K_1^{1/2}\sigma_Z
                                 B_1(t^{3-2\alpha}):t\ge0\bigr), 
\notag \bigl((1-\alpha)^{-1}\sigma_Z
                B_2(t):t\ge0\bigr),\\
                &\hskip2in 
                \bigl((1-\alpha)^{-1}\sigma_Z B_3(t):t\ge0\bigr)\Bigr) \nonumber 
\end{align*}
 because the difference between the two processes vanishes in the
 limit. For example, 
\[
n^{2\alpha-2}\zeta_n^{-2}\sigma_n^{-2}\sum_{i=1}^{[n^\beta
  t]}A_i\int_{-\infty}^\infty z\, G_{\zeta_nA_i}(dz)=O\left(n^{1-2\alpha}\right)=o(1)\,,
\]
and similarly with the other two components. Furthermore, for large $n$,
\begin{align*}
&S_n([n^\beta t],\delta,L)  \\
&= \sum_{i=[n^\beta\delta]}^{[n^\beta
                    L]-1}(A_{i+[n^\beta t]} -A_i)Y_{ni}
+\sum_{i=n-[n^\beta t]}^{n-1} (A_{i+[n^\beta t]}-
   A_{i+[n^\beta t]-n} -A_i)Y_{ni}  \\
  +&\sum_{i=n}^{n+[n^\beta L]}(A_{i+[n^\beta t]}-A_{i+[n^\beta t]-n}-A_i+A_{i-n})Y_{ni} 
=: V_n^1(t)+V_n^2(t)+V_n^3(t).
\end{align*}

Starting with  $V_n^3$, we write 
\begin{equation} \label{e:V3}
V_n^3(t)=n^{-(1-\alpha)(1-\beta)}\sum_{i=1}^{[n^\beta L]}f_n\left(n^{-\beta}i,t\right)\left(A_{n+i}-A_i\right)Y_{n,n+i}\,,
\end{equation}
where for $0\leq s\leq L$, 
\[
f_n(s,t)=n^{(1-\alpha)(1-\beta)}\frac{A_{n+[n^\beta s]+[n^\beta t]}-A_{[n^\beta s]+[n^\beta t]}-A_{n+[n^\beta s]}+A_{[n^\beta s]}}{A_{n+[n^\beta s]}-A_{[n^\beta s]}}\,.
\]
It is elementary that for  fixed $s,t$, as $n\to\infty$, 
\begin{align*}
A_{n+[n^\beta s]+[n^\beta t]}-A_{n+[n^\beta s]}&\ll A_{[n^\beta s]+[n^\beta t]} - A_{[n^\beta s]}\\
&\sim     (1-\alpha)^{-1}n^{\beta(1-\alpha)}\left[(s+t)^{1-\alpha}-s^{1-\alpha}\right]\,,   
\end{align*}
while $A_{n+[n^\beta s]}-A_{[n^\beta s]}\sim
(1-\alpha)^{-1}n^{1-\alpha}$. Therefore, 
\begin{equation}
\label{lpm.l5.eq5}\lim_{n\to\infty}f_n(s,t)=s^{1-\alpha}-(s+t)^{1-\alpha} =:f(s,t),
\end{equation}
and the limit is easily seen to be uniform in $0\leq s\leq L$ and $t$
in a compact interval. We will
show that, conditionally on $E_0$, 
\begin{align}
\label{lpm.l5.eq8}
&\bigl(n^{2\alpha-2}V_n^3(t), \, t\geq 0\bigr)\\
\notag&\Rightarrow
\left( \sigma_Z(1-\alpha)^{-1}\int_0^L
\left[s^{1-\alpha}-(s+t)^{1-\alpha}\right]dB_3(s),\, t\geq 0\right) 
\end{align}
in finite-dimensional distributions, as $n\to\infty$.  To this end,
set  
\[
c_{nj}(k,t)=\inf_{(j-1)L/k\le s\le jL/k}f_n(s,t), \, k\ge1, \, 1\le j\le k\,,
\]
and
\[
e_{ni}(k,t)= f_n\bigl(n^{-\beta}i, t\bigr) - c_{n,\lceil L^{-1}n^{-\beta}
  ki\rceil}(k,t)\geq 0, \, k\ge1,\, 1\le i\le[n^\beta L]\,.
\]
By \eqref{lpm.l5.eq5} and monotonicity, 
\begin{equation}
\label{l4.eq5}\lim_{n\to\infty}c_{nj}(k,t)=f\left((j-1)k^{-1}L,t\right),
\, 1\le j\le k\,. 
\end{equation}
A standard continuity argument shows that 
\begin{equation}
\label{l4.eq6}\lim_{k\to\infty}\limsup_{n\to\infty}\sup_{t\in A}\max_{1\le i\le[n^\beta L]}e_{ni}(k,t)=0
\end{equation}
for any compact set $A$. We have 
\begin{align*}
&\sum_{i=1}^{[n^\beta L]} c_{n,\lceil L^{-1}n^{-\beta} ki\rceil}(k,t)(A_{n+i}-A_i)Y_{n,n+i}\\
&=\sum_{j=1}^{k^\prime} c_{nj}(k,t)\sum_{i\in\bigl(k^{-1}Ln^\beta(j-1),k^{-1}Ln^\beta j\bigr]\cap\bbz}(A_{n+i}-A_i)Y_{n,n+i}\\
&=\sum_{j=1}^{k^\prime}c_{nj}(k,t)\Bigl(\xi_n^{3\circ}\bigl(
                                                                                                                                 k^{-1}Lj\bigr)-\xi_n^{3\circ}\bigl(  k^{-1}L(j-1)\bigr)\Bigr)
=:W_{nk}(t)\,,
\end{align*}
where $k^\prime = \lceil L^{-1}n^{-\beta} k[n^\beta L]\rceil$. 
This, together with \eqref{lpm.l4.eq1} and \eqref{l4.eq5}, implies  that for fixed $k$, as $n\to\infty$,
\begin{align}\label{e:W.nk}
&\bigl( n^{\alpha-\beta/2-1}W_{nk}(t), \, t\geq 0\bigr)\\
\notag \Rightarrow&\Biggl( (1-\alpha)^{-1}\sigma_Z\sum_{j=1}^k
     f\left((j-1)k^{-1}L,t\right)\left(B_3(k^{-1}jL)-B_3(k^{-1}(j-1)L)\right),  \\
\notag&\hskip3in  \, t\geq 0\Biggr)                                                        
\end{align}
in finite-dimensional distributions. We have 
\begin{align*}
&\sum_{i=1}^{[n^\beta
  L]}f_n\bigl(n^{-\beta}i,t\bigr)\left(A_{n+i}-A_i\right)Y_{n,n+i}-W_{nk}(t)\\
  &=\sum_{i=1}^{[n^\beta L]} e_{ni}(k,t)\left(A_{n+i}-A_i\right)Y_{n,n+i}\,.
\end{align*}

It follows from \eqref{lpm.newl2.eq2}  that, for large $n$, 
$$
\sup_{i,j\ge0:i\neq
  j}\left(A_i-A_{i-n}\right)\left(A_j-A_{j-n}\right)E\left(Y_{ni}Y_{nj}|E_0\right)\le0\,.
$$
This, along with \eqref{lpm.newl2.eq1} and the non-negativity of each
$e_{ni}$,  implies that for large $n$, 
\begin{align*}
&E\left(\left[\sum_{i=1}^{[n^\beta L]} e_{ni}(k,t)\left(A_{n+i}-A_i\right)Y_{n,n+i}\right]^2\Biggr|E_0\right)\\
&\le\sum_{i=1}^{[n^\beta L]} \left[e_{ni}(k,t)\left(A_{n+i}-A_i\right)\right]^2E(Y_{n,n+i}^2|E_0)\\
&=O\left(\max_{1\le j\le[n^\beta L]}e_{nj}(k,t)^2\sum_{i=1}^{[n^\beta L]}\left(A_{n+i}-A_i\right)^2\right)\\
&=O\left(n^{2-2\alpha+\beta}\max_{1\le j\le[n^\beta
L]}e_{nj}(k,t)^2\right)\,. 
\end{align*}
  Invoking \eqref{l4.eq6} we conclude that for any compact set $A$, 
\begin{align}\label{lpm.l4.eq8}
&\lim_{k\to\infty}\limsup_{n\to\infty}n^{2\alpha-\beta-2}\sup_{t\in
           A}E\Biggl[\Biggl(W_{nk}(t)\\
\notag           &\hskip1in-\sum_{i=1}^{[n^\beta L]}f_n\bigl(n^{-\beta}i,t\bigr)\left(A_{n+i}-A_i\right)Y_{n,n+i}\Biggr)^2\Biggr|E_0\Biggr]
=0\,.
\end{align}
 As $k\to\infty$, the process in the right hand side of \eqref{e:W.nk}
 converges in
finite-dimensional distributions to the process in the right-hand side
of \eqref{lpm.l5.eq8}.  Since
$(2\alpha-2)-(1-\alpha)(1-\beta)=\alpha-\beta/2-1$, the claim
\eqref{lpm.l5.eq8} follows from \eqref{e:V3} and \eqref{lpm.l4.eq8} 
by the ``convergence together'' argument; see Theorem 3.2 in
\cite{billingsley:1999}.

A nearly identical argument shows that, conditionally on $E_0$, 
\begin{align}\label{lpm.l4.eq9}
 \bigl(n^{2\alpha-2}V_n^2(t), \, t\geq 0\bigr)\Rightarrow&
\left( -\sigma_Z(1-\alpha)^{-1}\int_0^t
  (t-s)^{1-\alpha}dB_2(s),\, t\geq 0\right)  \\
 \notag \eid&  \left( \sigma_Z(1-\alpha)^{-1}\int_0^t
  (t-s)^{1-\alpha}dB_2(s),\, t\geq 0\right)
\end{align}
in finite-dimensional distributions.

The situation with the term $V_n^1$ is, once again, similar, with a
small twist. Since 
\[
\lim_{n\to\infty}\frac{A_{[n^\beta s]+[n^\beta t]} -A_{[n^\beta
    s]}}{A_{[n^\beta s]}}=\frac{(s+t)^{1-\alpha}-s^{1-\alpha}}{s^{1-\alpha}}
\]
uniformly for $\delta\le s\le L$ and $t$, our argument now shows that,
conditionally on $E_0$,  
\[
\bigl( n^{-(2-2\alpha)}V_n^1, \, t\geq 0\bigr)\
\Rightarrow
\left(\sigma_Z K_1^{1/2}\int_\delta^L  
\frac{(s+t)^{1-\alpha}-s^{1-\alpha}}{s^{1-\alpha}} M(ds), \, t\geq 0  \right)
\]
in finite-dimensional distributions, where $M$ is a centred Gaussian
random measure with the variance measure with the density
$(3-2\alpha)s^{2-2\alpha}$, $s>0$. Since the centred Gaussian random
measures $(1-\alpha)^{-1}B_3(ds)$ and $K_1^{1/2}M(ds)/s^{1-\alpha}$
have the same variance measure, this means that, conditionally on
$E_0$,  
\begin{align}\label{lpm.l5.eq9}
&\bigl(n^{2\alpha-2}V_n^2(t), \, t\geq 0\bigr)\\
\notag &\Rightarrow
\left( \sigma_Z(1-\alpha)^{-1}\int_\delta^L  
\bigl((s+t)^{1-\alpha}-s^{1-\alpha}\bigr)dB_3(s) ,\, t\geq 0\right)
\end{align}
in finite-dimensional distributions.

Since \eqref{lpm.l5.eq8}, \eqref{lpm.l4.eq9} and \eqref{lpm.l5.eq9}
are all consequences of \eqref{lpm.l4.eq1}, the convergence statements
they contain hold jointly, and jointly with  
$\zeta_n T_n^*\Rightarrow T_0$.   The claim \eqref{lpm.l4.eq3} follows.
\end{proof}

The next lemma treats the sequence of shifts appearing due to
conditioning on $E_0$.
\begin{lemma} \label{l:mu.n}
Define
\begin{align*}
&\mu_n(t)\\
\notag&=n^{2\alpha-2} \sum_{i=0}^\infty\left(A_{i+[n^\beta t]}-A_{i+[n^\beta
           t]-n}-A_i+A_{i-n}\right)   \int_{-\infty}^\infty
z\,  G_{\zeta_n(A_i-A_{i-n})}(dz),
\end{align*}
for $t\geq 0$ and $n\geq 1$. Then $\mu_n\rightarrow \mu_\infty$ as $n\to\infty$, in
$D([0,\infty))$ equipped with the Skorohod $J_1$ topology, where
$\mu_\infty(t) = -\vep t^{3-2\alpha}, \, t\geq 0$.
\end{lemma}
\begin{proof}
  Writing
  \begin{align*}
\mu_n(t)
=&n^{2\alpha-2} \zeta_n\sum_{i=0}^\infty\left(A_{i+[n^\beta t]}-A_{i+[n^\beta
           t]-n}-A_i+A_{i-n}\right) \bigl( A_i-A_{i-n}\bigr)\\
+&n^{2\alpha-2} \sum_{i=0}^\infty\left(A_{i+[n^\beta t]}-A_{i+[n^\beta
           t]-n}-A_i+A_{i-n}\right)  \\
&\hskip 0.2in \left[ \int_{-\infty}^\infty
z\,  G_{\zeta_n(A_i-A_{i-n})}(dz)-\zeta_n\bigl( A_i-A_{i-n}\bigr)\right]\\
&=: \mu_n^{(1)}(t)+\mu_n^{(2)}(t), \, t\geq 0, 
  \end{align*}
 the claim of the lemma will follow once we prove that 
\begin{equation}  \label{e:mu.1}
\mu_n^{(1)}\rightarrow \mu_\infty  \ \ \text{in} \ \ 
D([0,\infty))
\end{equation}
and
\begin{equation}  \label{e:mu.2}
\mu_n^{(2)}(t) \to 0 \ \ \text{uniformly on compact intervals.}
\end{equation} 

We start by proving \eqref{e:mu.2}. Fix $L>0$ so that $0\leq t\leq
L$. Suppose first that
$1/2<\alpha<5/6$. By \eqref {e:tilt.facts}
\begin{align*}
&\bigl| \mu_n^{(2)}(t)\bigr| \\
=& O\left( 
n^{2\alpha-2} \zeta_n^2\sum_{i=0}^\infty\left|A_{i+[n^\beta t]}-A_{i+[n^\beta
           t]-n}-A_i+A_{i-n}\right| \bigl( A_i-A_{i-n}\bigr)^2\right)\\
=& O\left( 
n^{2\alpha-2} \zeta_n^2n^\beta \sum_{i=1}^\infty i^{-\alpha} \bigl(
  A_i-A_{i-n}\bigr)^2\right) 
=O\left( 
n^{2\alpha-2} \zeta_n^2n^\beta n^{3-3\alpha} \right)\to 0
\end{align*}
uniformly in $0\leq t\leq L$, showing \eqref{e:mu.2}. On the other
hand, if $\alpha\geq 5/6$, then $\kappa\geq 3$ in \eqref{e:moment.eq},
so by \eqref {e:tilt.facts}
\begin{align*}
&\bigl| \mu_n^{(2)}(t)\bigr| \\
=& O\left( 
n^{2\alpha-2} \zeta_n^3\sum_{i=0}^\infty\left|A_{i+[n^\beta t]}-A_{i+[n^\beta
           t]-n}-A_i+A_{i-n}\right| \bigl( A_i-A_{i-n}\bigr)^3\right)\\
=& O\left( 
n^{2\alpha-2} \zeta_n^3n^\beta \sum_{i=1}^\infty i^{-\alpha} \bigl(
  A_i-A_{i-n}\bigr)^3\right) 
=O\left( 
n^{2\alpha-2} \zeta_n^3n^\beta n^{4-4\alpha} \right)\to 0
\end{align*}
uniformly in $0\leq t\leq L$, again showing \eqref{e:mu.2}. 

We now prove \eqref{e:mu.1}. The pointwise
convergence is clear: for fixed $t$, 
\begin{align*}
\mu_n^{(1)}(t)&=
          \sigma_Z^2\sigma_n^{-2}n^{2\alpha-1}\vep\sum_{i=0}^\infty\left(A_{i+[n^\beta
          t]}-A_{i+[n^\beta
          t]-n}\right)(A_i-A_{i-n})-n^{2\alpha-1}\vep\\
          & \to-\vep t^{3-2\alpha}
\end{align*}
as $n\to\infty$, where we have used \eqref{lpm.l0.eq4}. Next, 
as in \eqref{lpm.l0.eq11} we can write for $t\geq 0$, 
\begin{align*}
 \mu_n^{(1)}(t)=& \frac{n^{2\alpha-2}\zeta_n}{2} 
\Biggl[\sum_{i=0}^{n-1}\left(A_i-A_{i-[n^\beta t]}\right)^2
\\
&\hskip1in  +\sum_{i=n-[n^\beta t]}^\infty
\left(A_{i+[n^\beta t]}-A_{i+[n^\beta t]-n}-A_i+A_{i-n}\right)^2\Biggr]\\
&=: \mu_n^{(11)}(t)+\mu_n^{(12)}(t). \notag 
\end{align*}
The claim \eqref{e:mu.1} will follow once we show that both
$\mu_n^{(11)}$ and $\mu_n^{(12)}$ converge in
$D([0,\infty))$ to continuous limits (both constant factors of
$\mu_\infty$). The fact that $\mu_n^{(11)}$ converges pointwise to a
constant factor of of the pointwise limit of $\mu_n^{(1)}$ is an
intermediate step in the proof of \eqref{lpm.l0.eq4}. Since
$\mu_n^{(11)}$ is a monotone function, its convergence in
$D([0,\infty))$ follows. 

We already know that $\mu_n^{(12)}$ converges pointwise to a
continuous limit. Let $i_0$ be such that $a_i$ is monotone for $i\geq
i_0$. Write for $t\geq 0$
\begin{align*}
\mu_n^{(12)}(t)=& \frac{n^{2\alpha-2}\zeta_n}{2} 
\Biggl[ \sum_{i=n+i_0}^\infty
\left(A_{i+[n^\beta t]}-A_{i+[n^\beta t]-n}-A_i+A_{i-n}\right)^2 \\
&\hskip 0.5in -\sum_{i=n-[n^\beta t]}^{n+i_0-1}
\left(A_{i+[n^\beta t]}-A_{i+[n^\beta t]-n}-A_i\right)^2\Biggr] \\
&=: \mu_n^{(121)}(t)-\mu_n^{(122)}(t),
\end{align*}
so it is enough to show that both $\mu_n^{(121)}$ and $\mu_n^{(122)}$
converge in $D([0,\infty))$ to continuous limits. Splitting further,
we write for $t\geq 0$,
\begin{align*}
  \mu_n^{(122)}(t)&= \frac{n^{2\alpha-2}\zeta_n}{2} 
\Biggl[ \sum_{i=n-[n^\beta t]}^{n+i_0-1} A_{i+[n^\beta t]-n}^2 \\
&+ \sum_{i=n-[n^\beta t]}^{n+i_0-1} \bigl( A_i-A_{i+[n^\beta t]}\bigr)
\bigl(  A_i-A_{i+[n^\beta t]}-2A_{i+[n^\beta t]-n}\bigr)\Biggr]\\
&=:  \mu_n^{(1221)}(t)+  \mu_n^{(1222)}(t).
\end{align*}
Clearly,
\begin{align*}
\mu_n^{(1221)}(t)=\frac{n^{2\alpha-2}\zeta_n}{2} \sum_{i=0}^{[n^\beta
  t]+i_0-1} A_i^2 
\end{align*}
converges pointwise to a constant factor of $\mu_\infty$. Since
$\mu_n^{(1221)}$ is monotone, we conclude that $\mu_n^{(1221)}$
converges in $D([0,\infty))$ to a continuous limit. In order to prove
that so does $\mu_n^{(122)}$, we will show that $\mu_n^{(1222)}(t)\to 0$
uniformly on compact intervals. Considering once again $0\leq t\leq
L$, we have
\begin{align*}
&\bigl| \mu_n^{(1222)}(t)\bigr|\\
 &\leq \frac{n^{2\alpha-2}\zeta_n}{2}  
 \sum_{i=n-[n^\beta t]}^{n+i_0-1} 
\bigl( A_{i+[n^\beta t]}-A_i\bigr)
\bigl[\big(  A_{i+[n^\beta t]}-A_i\bigr) +2A_{i+[n^\beta t]-n}\bigr]\\
&=O\left( n^{2\alpha-2}\zeta_n \sum_{i=n-[n^\beta t]}^{n+i_0-1} n^\beta
  n^{-\alpha} \bigl(  n^\beta
  n^{-\alpha} + n^{\beta(1-\alpha)}\bigr)\right)\\
  &=O\left(
         n^{\alpha-2}\zeta_n
             n^{3\beta-\beta\alpha}\right) \to 0
\end{align*}
uniformly over $0\leq t\leq L$, as required.

Finally, we already know that $\mu_n^{(121)}$ converges pointwise to a
continuous limit. Furthermore, by the choice of $i_0$, $\mu_n^{(121)}$
is a monotone function. Therefore, it converges in $D([0,\infty))$,
and the proof is complete. 
\end{proof}
  
The following is the final lemma before we prove Theorem
\ref{lpm.t1}. 

\begin{lemma}\label{lpm.l6}
 Suppose that \eqref{e:moment.eq}  and\eqref{eq.chf} hold. Let 
\begin{equation}
\label{lpm.l6.eq0}S_n(j)=\sum_{i=j}^{j+n-1}X_i, \, j\ge0, \, n\geq 1.
\end{equation}
As $n\to\infty$, conditionally on $E_0$,
\begin{align*}
& \left( n^{-(2-2\alpha)} \left( S_n ( [n^\beta t] ) - n \vep \right),
                 \, t \ge 0\right)\\
& \Rightarrow  \left( (2C_\alpha)^{1/2} B_H(t) + \vep^{-1} C_\alpha \sigma_Z^2T_0 - \vep
               t^{3-2\alpha} , \, t\ge0 \right)
\end{align*}
in finite-dimensional distributions, 
where $(B_H(t):t\ge0)$ is the standard fractional Brownian motion
\eqref{lpm.fbm} with the Hurst exponent $H$ given in
\eqref{lpm.eq.defH}, $C_\alpha$ is the constant defined in \eqref{lpm.defC},  
 and  $T_0$ is a standard exponential random variable independent of
 the fractional Brownian motion. 
\end{lemma}
\begin{proof}
It follows from \eqref{lpm.newl2.eq2} and the eventual monotonicity of
the sequence $(A_n)$ that there is $i_0\ge 0$ such that for all large
$n$, 
\begin{equation}
\label{lpm.l6.eq1}\sup_{i_0\le i<j}E\left(Y_{ni}Y_{nj}|E_0\right)\le0\,.
\end{equation}
For  fixed $L, t>0$ this and \eqref{lpm.newl2.eq1} imply that 
\begin{align*}
& E\left[\left(\sum_{i=[n^\beta L]}^{n-[n^\beta t]-1}\left(A_{i+[n^\beta
  t]}-A_i\right)Y_{ni}\right)^2\Biggr|E_0\right]\\
  =&O\left(\sum_{i=[n^\beta
  L]}^\infty\left(A_{i+[n^\beta t]}-A_i\right)^2\right)\\
=&O\left(\sum_{j=[n^\beta
   L]}^\infty\left(\left(j+[n^\beta
   t]\right)^{1-\alpha}-j^{1-\alpha}\right)^2\right) \\
   \leq & O\left(n^{4-4\alpha}\int_L^\infty \left[(x+t)^{1-\alpha}-x^{1-\alpha}\right]^2dx\right)\,.
\end{align*}
Therefore, for fixed $t$,
\begin{equation}\label{lpm.l6.eq2}
\lim_{L\to\infty}\limsup_{n\to\infty}E\left[\left(n^{2\alpha-2}\sum_{i=[n^\beta L]}^{n-[n^\beta t]-1}\left(A_{i+[n^\beta t]}-A_i\right)Y_{ni}\right)^2\Biggr|E_0\right]=0\,.
\end{equation}

Since the sequence $(a_n)$ is eventually monotone, we can increase, if
necessary, $i_0$ to guarantee that  $A_{j+k}-A_{j}\le
A_{i+k}-A_i$ for all $i_0\le i\le j$ and $k\ge0$. By
\eqref{lpm.l6.eq1}, for fixed $L,t>0$, large $n$ and $i,j\ge n+[n^\beta
L]$, 
\begin{align*} 
&\left(A_{i+[n^\beta t]}-A_{i+[n^\beta t]-n}-A_i+A_{i-n}\right) \\
&\hskip1in \left(A_{j+[n^\beta t]}-A_{j+[n^\beta  t]-n}-A_j+A_{j-n}\right)E\left(Y_{ni}Y_{nj}\bigr|E_0\right)\le0\,,
\end{align*}
and the same argument as above implies that 
\begin{align} \label{e:miss.lim}
&\lim_{L\to\infty}\limsup_{n\to\infty}E\Biggl[
 \Bigl(n^{2\alpha-2}\\
\notag &\hskip0.1in \sum_{i=n+[n^\beta L]+1}^\infty\left(A_{i+[n^\beta t]}-A_{i+[n^\beta t]-n}-A_i+A_{i-n}\right)Y_{ni}\Bigr)^2\Biggr|E_0\Biggr]=0\,.
\end{align}
Similarly, for a fixed $t>0$, 
\begin{equation} \label{e:miss.lim2}
\lim_{\delta\to0}\limsup_{n\to\infty}E\left[\left(n^{2\alpha-2}\sum_{i=i_0}^{[n^\beta \delta]-1}\left(A_{i+[n^\beta t]}-A_i\right)Y_{ni}\right)^2\Biggr|E_0\right]=0\,,
\end{equation}
and it is elementary that for a fixed $t>0$, 
\begin{equation}
\label{lpm.l6.eq3}\lim_{n\to\infty}E\left[ \left(n^{2\alpha-2}\sum_{i=0}^{i_0-1}\left(A_{i+[n^\beta t]}-A_i\right)Y_{ni}\right)^2\Biggr|E_0\right]=0\,.
\end{equation}

It follows from  \eqref{lpm.l6.eq2}, \eqref{e:miss.lim},
\eqref{e:miss.lim2}, 
 \eqref{lpm.l6.eq3} and Lemma \ref{lpm.l4} that, conditionally on $E_0$,  
\begin{align}\label{lpm.l7.eq1}
&\left[   \zeta_nT_n^*,\left(n^{-(2-2\alpha)}\sum_{i=0}^\infty
\left(A_{i+[n^\beta  t]}-A_{i+[n^\beta t]-n}-A_i+A_{i-n}\right)Y_{ni}, 
\, t\ge0\right)\right]\\
\nonumber&\Rightarrow \left[ T_0,\left( (1-\alpha)^{-1}\sigma_Z \left( \int_0^\infty\left[(s+t)^{1-\alpha}-s^{1-\alpha}\right] dB_1(s)\right. \right.\right.\\
\nonumber&\,\,\,\,\,\, \,\,\,\,
           +\left.\left. \left.\int_0^t(t-s)^{1-\alpha}\,dB_2(s)
           +\int_0^\infty
           \left[(s+t)^{1-\alpha}-s^{1-\alpha}\right]dB_3(s)\right),
           \, t\ge0\right)\right]\,,
\end{align}
in finite-dimensional distributions,  as $n\to\infty$. Furthermore,
one can easily check the Lindeberg conditions of the central limit
theorem to see that 
\begin{align}
\label{lpm.l7.eq2}&\left(n^{-(2-2\alpha)}\sum_{i=-[n^\beta t]}^{-1}
                    A_{i+[n^\beta t]} Z_{n-1-i}, \, t\ge0\right)\\
\notag & \Rightarrow
           \left((1-\alpha)^{-1}\sigma_Z\int_0^t(t-s)^{1-\alpha}\, dB_0(s),
           \, t\ge0\right)
\end{align}
in finite-dimensional distributions, as $n\to\infty$, 
where $B_0$ is a standard Brownian motion. Note that the random
variables in the left hand side of \eqref{lpm.l7.eq2} are independent
of the the random
variables in the left hand side of \eqref{lpm.l7.eq1} and, in
particular, independent of $E_0$.

Using \eqref{lpm.l0.eq8} we conclude by \eqref{lpm.l7.eq1} and
\eqref{lpm.l7.eq2} that,  in the notation of Lemma \ref{l:mu.n},
conditionally on $E_0$,  
\begin{align*}\label{lpm.l7.eq3}
&\left[   \zeta_nT_n^*,\left(n^{-(2-2\alpha)}\left(S_n([n^\beta
                                    t])-S_n\right)-\bigl( 1+\zeta_n^{-2}\sigma_n^{-2}\bigr)\mu_n(t), \, t\ge0\right)\right]\\
\nonumber&\Rightarrow \left[ T_0,\left( (1-\alpha)^{-1}\sigma_Z \left( 
\int_0^t(t-s)^{1-\alpha}\,dB_0(s)\right.\right.\right.\\
\nonumber&\,\,\,\,\,\, \,\,\,\,
+\int_0^\infty\left[(s+t)^{1-\alpha}-s^{1-\alpha}\right] dB_1(s)\\
\nonumber&\,\,\,\,\,\, \,\,\,\,
           +\left.\left. \left.\int_0^t(t-s)^{1-\alpha}\,dB_2(s)
           +\int_0^\infty
           \left[(s+t)^{1-\alpha}-s^{1-\alpha}\right]dB_3(s)\right),
           \, t\ge0\right)\right]\\
&\eid\left[
                                      T_0,\left(2^{1/2}(1-\alpha)^{-1}\sigma_Z\int_{-\infty}^\infty \left[(t-s)_+^{1-\alpha}-(-s)_+^{1-\alpha}\right]dW(s), \,
t\ge0\right)\right]\nonumber
\end{align*}
in finite-dimensional distributions as $n\to\infty$, where at the
intermediate step the four standard Brownian motions, $B_0,B_1,B_2$
and $B_3$ are independent (and independent of $T_0$), and in the final
expression $(W(s), \, s\in\bbr)$ is a two-sided standard Brownian motion,
independent of $T_0$. By  \eqref{lpm.l0.eq12}, this can be restated as
saying that, conditionally on $E_0$, 
\begin{align*}
&\left[   \zeta_nT_n^*,\left(n^{-(2-2\alpha)}
\left(S_n([n^\beta t])-S_n\right)-\mu_n(t), \, t\ge0\right)\right]\\
&\Rightarrow\left[    T_0,\left( (2C_\alpha)^{1/2} B_H(t), \,  t\ge0 \right)\right],
\end{align*}
and by Lemma
\ref{l:mu.n} also
\begin{align*}
&\left[   \zeta_nT_n^*,\left(n^{-(2-2\alpha)}
\left(S_n([n^\beta t])-S_n\right), \, t\ge0\right)\right]\\
&\Rightarrow\left[    T_0,\left( (2C_\alpha)^{1/2} B_H(t)-\vep t^{3-2\alpha}, \,  t\ge0 \right)\right]
\end{align*}
in finite-dimensional distributions, as $n\to\infty$. Since
\begin{align*}
  n^{-(2-2\alpha)} \bigl(S_n ( [n^\beta t] ) - n \vep\bigr)
= n^{-(2-2\alpha)}\bigl(S_n([n^\beta t])-S_n\bigr) +
  \bigl(n^{2\alpha-2}\zeta_n^{-1}\bigr)\zeta_nT_n*, 
\end{align*}
the claim of the lemma follows from the definition \eqref{eq.defzetan}
of $\zeta_n$ and \eqref{lpm.l0.eq3}.  
\end{proof}

Now we are in a position to prove Theorem \ref{lpm.t1}.
\begin{proof}[Proof of Theorem \ref{lpm.t1}]  
We will prove that 
\begin{equation}
\label{t1.eq1}\biggl\{P\left[\left( n^{-(2-2\alpha)} \left( S_n (
        [n^\beta t] ) - n \vep \right), \, 
 0\le t<\infty\right)\in\cdot\Bigr|E_0\right], \, n\ge1\biggr\}
\end{equation}
is a tight family of probability measures on $D([0,\infty))$ equipped
with the Skorohod $J_1$ topology. Assuming for a moment that this is
true, it would follow from Lemma \ref{lpm.l6} that, conditionally on $E_0$,
\begin{align*}
& \left( n^{-(2-2\alpha)} \left( S_n ( [n^\beta t] ) - n \vep \right) : t \ge 0\right)\\
& \Rightarrow  \left( (2C_\alpha)^{1/2} B_H(t) + \vep^{-1} C_\alpha\sigma_Z^2 T_0 - \vep t^{3-2\alpha} : t\ge0 \right)
\end{align*}
weakly in $D([0,\infty))$, as $n\to\infty$. Since the  functional
  $  \bx\mapsto \inf\{ t\geq 0:\, x(t)\leq 0\}$  on $D([0,\infty))$
  is, clearly, a.s. continuous with 
respect to the law induced on that space by the limiting process, the
continuous mapping theorem would imply that,  conditionally on $E_0$,
 \begin{align*}
\nonumber n^{-\beta} I_n(\vep)&= \inf\left\{t\ge0:n^{-(2-2\alpha)}\left( S_n([n^\beta t])-n\vep \right) \le0\right\}\\
&\Rightarrow\inf\left\{t\ge0: (2C_\alpha)^{1/2} B_H(t) 
+ \vep^{-1}   C_\alpha \sigma_Z^2T_0 - \vep t^{3-2\alpha} \le   0\right\} =\tau_\vep
\end{align*}
as $n\to\infty$.  Therefore, establishing tightness of the family
\eqref{t1.eq1} suffices to complete the proof of Theorem
\ref{lpm.t1}, and by Lemma \ref{l:mu.n}  it is enough to prove that the
family
\begin{equation}
\label{t1.eq1aa}\biggl\{P\left[\left( n^{-(2-2\alpha)} \left( S_n (
        [n^\beta t] ) - n \vep \right)-\mu_n(t), \, 
 0\le t<\infty\right)\in\cdot\Bigr|E_0\right], \, n\ge1\biggr\}
\end{equation}
is a tight family of probability measures on $D([0,\infty))$.

We have to prove tightness of the restriction of the family
\eqref{t1.eq1aa} to the interval $[0,L]$ for any $L>0$, so fix $L$. We
start by showing that  
\begin{align}\notag 
&E\left[\left(S_n\left([n^\beta t]\right)-n^{2\alpha-2}\mu_n(t)
-S_n\left([n^\beta s]\right)+n^{2\alpha-2}\mu_n(s)\right)^2\biggr|E_0\right]\\
\label{lpm.al8.eq1}&=O\left(\left([n^\beta t]-[n^\beta s]\right)^{3-2\alpha}\right)\,,
\end{align}
uniformly for $0\le s\le t\le L$.   We write 
\begin{align*}
&S_n\left([n^\beta t]\right)-n^{2\alpha-2}\mu_n(t)
-S_n\left([n^\beta s]\right)+n^{2\alpha-2}\mu_n(s) \\
&= \sum_{i=-[n^\beta t]}^{-1} \left(A_{i+[n^\beta t]} -A_{i+[n^\beta
  s]}\right)Z_{n-i-1} \\
 &+ \sum_{i=0}^\infty\left(A_{i+[n^\beta t]}
-A_{i+[n^\beta   t]-n}-A_{i+[n^\beta   s]}
+A_{i+[n^\beta   s]-n}\right)Y_{ni}. 
\end{align*}
Since $Z_n,Z_{n+1},\ldots$ are independent of $E_0$, by  Lemma
\ref{lpm.newl2}, 
\begin{align*}
& E\left[\left(S_n\left([n^\beta t]\right)-n^{2\alpha-2}\mu_n(t)-S_n\left([n^\beta s]\right)+n^{2\alpha-2}\mu_n(s)\right)^2\biggr|E_0\right]\\
 =& O\Biggl[\sum_{j=0}^{[n^\beta t]-1}\left(A_j-A_{j+[n^\beta
           s]-[n^\beta t]}\right)^2 \\
&\hskip1in + \sum_{i=0}^\infty\left(A_{i+[n^\beta
           t]}-A_{i+[n^\beta t]-n}  
-A_{i+[n^\beta s]}+A_{i+[n^\beta
           s]-n}\right)^2\Biggr] \\
 =& O\left(\left([n^\beta t]-[n^\beta
           s]\right)^{3-2\alpha}\right) 
\end{align*}
 uniformly for $0\le s\le t\le L$ by
\eqref{lpm.newl0.eq2} with $\kappa=2$, and \eqref{lpm.al8.eq1}
follows.

Let now $0\le r\le s\le t\le L$. If $t-r\leq n^{-\beta}$, then 
\begin{align*}
&E\Biggl[\left|S_n([n^\beta s])-\mu_n(s)-S_n([n^\beta r])+\mu_n(r)\right|\\
&\hskip1in \left|S_n([n^\beta t])-\mu_n(t)-S_n([n^\beta s])+\mu_n(s)\right|\Bigr|E_0\Biggr]
\end{align*}
vanishes. On the other hand, if $t-r> n^{-\beta}$, then by
\eqref{lpm.al8.eq1} and the Cauchy-Schwarz inequality, the conditional
expectation can be bounded by 
\begin{align*}
 O\left(\left([n^\beta t]-[n^\beta r]\right)^{3-2\alpha}\right)
= O\left(n^{4-4\alpha}(t-r)^{3-2\alpha}\right)
\end{align*}
uniformly for $0\le r\le s\le t\le L$.  Since $3-2\alpha>1$, the
required tightness of the family in \eqref{t1.eq1aa} follows, which completes the proof of Theorem \ref{lpm.t1}.
\end{proof}

\section{Some useful facts} \label{sec:useful.f}

We collect in this section for easy reference a number of known or
easily derivable results.

The following integral evaluation follows from (2), (6) and (51) in
\cite{pickard:2011}. If $H\in(0,1)$, $H\not=1/2$, then
\begin{equation} \label{e:pickard}
\int_0^\infty \left[x^{H-1/2}-(x-1)_+^{H-1/2}\right]^2dx 
 =\frac{\cos(\pi H)\Gamma(2-2H)}{\pi H(1-2H)}{\Gamma(H+1/2)}^2\,.
\end{equation}

\bigskip

Next, we will need the following version of  the Berry-Essen
theorem valid for independent not necessarily identically
distributed summands; see \cite{batirov:manevich:nagaev:1977}.

Let $X_1,\ldots,X_n$ be independent zero mean random variables with
finite third moments. Denote 
\[
A=\sum_{i=1}^nE|X_i^3|, \ B=\sqrt{\sum_{i=1}^nE(X_i^2)}. 
\]
Assuming $B>0$ we have 
\begin{equation} \label{e:BE.nonid}
\left|P\left(\sum_{i=1}^nX_i\le Bz\right)-\Phi(z)\right|\le
C_uAB^{-3}, \, z\in\bbr\,,
\end{equation}
with $C_u$ a universal constant, and $\Phi$ the standard
normal CDF. The fact that the constant is universal means that
\eqref{e:BE.nonid} remains valid for $n=\infty$ as long the series in
the left hand side converges and $A,B$ are finite.

\bigskip

The following generalization of the Riemann-Lebesgue lemma can be
proven in the same way as the original statement. 
If $f:\bbr\to\bbr$ is a measurable function such that for some $\delta>0$,
\[
\int_{-\infty}^\infty e^{\theta x}|f(x)|dx<\infty\ \text{ for all }\theta\in[-\delta,\delta]\,,
\]
then
\begin{equation} \label{f.rl} 
\lim_{t\to\infty}\sup_{|\theta|\le\delta}\left|\int_{-\infty}^\infty
  e^{(\theta+it) x}f(x)\, dx\right|=0\,.
\end{equation} 

\bigskip

We will use a simple bound on the characteristic function $\phi$ of 
a random variable $X$ with a finite third moment. Let  $X'$ be an
independent copy of $X$ and $Y=X-X'$. Using the bound $\cos t \leq
1-t^2/2+|t|^3/6$ for $t\in\bbr$, we have 
 \begin{align*}
Ee^{itY}&\le1-t^2E(Y^2)/2+|t|^3E|Y|^3/6\\
&\le 1-t^2\Var(X)+4|t|^3E|X|^3/3\,.
\end{align*}
This implies that 
\begin{equation} \label{f.chf} 
|\phi(t)|\le\left(1-t^2\Var(X)+4|t|^3E|X|^3/3\right)^{1/2},\, t\in\bbr\,.
\end{equation}

\section*{Acknowledgment} AC is thankful to Rudra Sarkar for helpful discussions.


\begin{thebibliography}{5}
\expandafter\ifx\csname natexlab\endcsname\relax\def\natexlab#1{#1}\fi

\bibitem[Batirov et~al.(1977)Batirov, Manevich and
  Nagaev]{batirov:manevich:nagaev:1977}
{\sc K.~Batirov, D.~Manevich {\rm and} S.~Nagaev} (1977): The Esseen inequality
  for sums of a random number of differently distributed random variables.
\newblock {\em Mathematical Notes of the Academy of Sciences of the USSR\/}
  22:569--571.

\bibitem[Billingsley(1999)]{billingsley:1999}
{\sc P.~Billingsley} (1999): {\em Convergence of Probability Measures\/}.
\newblock Wiley, New York, 2nd edition.

\bibitem[Chakrabarty and Samorodnitsky(2022)]{chakrabarty:samorodnitsky:2022}
{\sc A.~Chakrabarty {\rm and} G.~Samorodnitsky} (2022): Clustering of large
  deviations in moving average processes: the short memory regime.
\newblock arXiv 2208.04582.

\bibitem[Pickard(2011)]{pickard:2011}
{\sc G.~Pickard} (2011): Representation formulae for the Fractional Brownian
  motion.
\newblock In {\em S\'eminaire de Probabilit\'es XLIII\/},  C.~Donati-Martin,
  A.~Lejay {\rm and} A.~Rounalt, editors, number 2006 in Lecture Notes in
  Mathemtics. Springer, Berlin, pp. 3--70.

\bibitem[Samorodnitsky(2016)]{samorodnitsky:2016}
{\sc G.~Samorodnitsky} (2016): {\em Stochastic Processes and Long Range
  Dependence\/}.
\newblock Springer, Cham, Switzerland.

\end{thebibliography}

\end{document}